\pdfoutput=1
\documentclass[a4paper]{amsart}
\usepackage[hmargin=3cm,vmargin=3.5cm]{geometry}
\usepackage{amsmath,amsthm,amssymb,amscd}
\usepackage[usenames,dvipsnames]{xcolor}
\usepackage[abs]{overpic}
\definecolor{darkgreen}{rgb}{0,0.5,0}
\definecolor{darkblue}{rgb}{0,0,0.5}
\usepackage[colorlinks=true,urlcolor=darkblue,citecolor=darkgreen,linkcolor=darkblue]{hyperref}
\usepackage{graphicx}
\usepackage{caption,subcaption}
\usepackage{esint}
\usepackage[all,cmtip]{xy}
\usepackage[latin1]{inputenc}
\usepackage[T1]{fontenc}
\usepackage{lmodern}
\usepackage{url}
\usepackage{tikz}
\usepackage[draft]{fixme}
\usepackage{youngtab}
\usepackage{bbm}
\usepackage{todonotes}

\newcommand{\boxdiagram}{{\fontsize{4pt}{12pt}\selectfont {\yng(1)}}}

\renewcommand{\bar}{\overline}

\renewcommand{\epsilon}{\varepsilon}

\renewcommand{\phi}{\varphi}
\renewcommand{\mod}{\ \mathrm{mod}\,\,}

\newcommand{\abs}[1]{\lvert #1 \rvert}

\DeclareMathOperator{\spn}{span}          

\newcommand{\bbC}{\mathbb{C}}

\newcommand{\bbZ}{\mathbb{Z}}
\newcommand{\bbN}{\mathbb{N}}

\newcommand{\bbR}{\mathbb{R}}

\newcommand{\setC}{\mathbb{C}}

\newcommand{\SU}{\mathrm{SU}}
\newcommand{\SO}{\mathrm{SO}}

\newcommand{\isom}{\cong}

\renewcommand{\Im}{\mathrm{Im}}

\DeclareMathOperator{\GL}{GL}
\DeclareMathOperator{\U}{U}

\DeclareMathOperator{\sr}{sr}

\def\XXint#1#2#3{{\setbox0=\hbox{$#1{#2#3}{\int}$}
\vcenter{\hbox{$#2#3$}}\kern-.5\wd0}}

\newcommand{\id}{\mathrm{id}}

\newcommand{\Hom}{\mathrm{Hom}}

\newcommand{\MCG}{M}
\newcommand{\MCGrand}{\bar{M}}

\newtheorem{thm}{Theorem}[section]

\newtheorem{prop}[thm]{Proposition}
\newtheorem{cor}[thm]{Corollary}

\newtheorem{lem}[thm]{Lemma}
\newtheorem{lemma}[thm]{Lemma}

\theoremstyle{remark}
\newtheorem{remark}[thm]{Remark}
\newtheorem{rem}[thm]{Remark}

\theoremstyle{definition}
\newtheorem{definition}[thm]{Definition}

\newtheorem{example}[thm]{Example}
\newcommand{\JK}[2][\empty]{\ifthenelse{\equal{#1}{\empty}}{\todo[color=green!30]{#2}}{\todo[color=green!30,#1]{#2}}}
\newcommand{\Soeren}[2][\empty]{\ifthenelse{\equal{#1}{\empty}}{\todo[color=yellow!30]{#2}}{\todo[color=yellow!30,#1]{#2}}}
\newcommand{\jk}[2][\empty]{\ifthenelse{\equal{#1}{\empty}}{\todo[color=red!30]{#2}}{\todo[color=red!30,#1]{#2}}}

\newcommand{\diagramrep}{\eta}
\begin{document}
\title[The homological content of the Jones representations at $q = -1$]{The homological content of\\the Jones representations at $q = -1$}

\author{Jens Kristian Egsgaard}
  \address{Centre for Quantum Geometry of Moduli Spaces\\ Faculty of Science\\
  Aarhus University\\
  DK-8000 Aarhus C, Denmark}
  \email{jk@qgm.au.dk}

\author{S{\o}ren Fuglede J{\o}rgensen}
  \address{Department of Mathematics\\
  Box 480\\
  Uppsala University\\
  SE-75106 Uppsala, Sweden}
  \email{soren.fuglede.jorgensen@math.uu.se}
	
\thanks{The authors have received funding from the Danish National Research Foundation grant DNRF95 (Centre for Quantum Geometry of Moduli Spaces - QGM). S.F.J. is furthermore supported by the Swedish Research Council Grant 621--2011--3629.}

\begin{abstract}
  We generalize a discovery of Kasahara and show that the Jones representations of braid groups, when evaluated at $q = -1$, are related to the action on homology of a branched double cover of the underlying punctured disk.
	
	As an application, we prove for a large family of pseudo-Anosov mapping classes a conjecture put forward by Andersen, Masbaum, and Ueno \cite{AMU} by extending their original argument for the sphere with four marked points to our more general case.
\end{abstract}
\maketitle

\tableofcontents
\section{Introduction}

The present paper is concerned with an interpretation of the $q$-dependent two-row Jones representations of braid groups -- discovered in 1983 by Vaughan Jones \cite{Jon83} -- at $q = -1$, in terms of the action of the braid group on the homology of certain double covers of punctured disks and spheres.

Jones' discovery \cite{Jon83} that the representations may be used to define interesting invariants of links, together with Witten's discovery \cite{WitJones} of the relation between these link invariants and the Chern--Simons theory of $3$-manifolds, has served as an inspiration for several important mathematical innovations, occasionally collected under the label ``quantum topology''. Any attempt to include here a full account on these historical developments would necessarily be lacking, and we shall make no such; indeed, for the experts, the Jones representation needs no introduction.

\subsection{The Jones representation and homology}
Let $B_n$ denote the braid group on $n$ strands, let $\MCGrand(g,n)$ denote the mapping class group of a genus $g$ surface $\Sigma_g^n$ with $n$ boundary components, and let $\MCG(g,m)$ denote the mapping class group of a genus $g$ surface with $m$ punctures.

Following the notation of Wenzl \cite[Thm.~2.2]{Wen}, the Jones representation defines for every Young diagram $\lambda = (\lambda_1 \geq \lambda_2 \geq \cdots \geq \lambda_p)$ with $n$ boxes, $\sum_j \lambda_j = n$, a representation $\pi_\lambda$ of $B_n$, depending on a parameter $q$ which we shall take to be a non-zero complex number which is furthermore assumed to be $n$-regular, i.e. it is not an $l$'th root of unity for any $2 \leq l \leq n$. We shall be dealing only with two-row Young diagrams and let $\pi^{n,d}_q$ denote the representation of $B_n$ obtained from $\lambda = (\lambda_1 \geq \lambda_1 -d)$, $d \geq 0$.

Let now $n$ be given, and let $g = n-1$. There then is a homomorphism $\Psi : B_{2n} \to \MCG(g,0)$ given by mapping the standard braid generators $\sigma_1, \dots, \sigma_{2n-1}$ to the (right) Dehn twists about the curves $\gamma_0, \beta_1, \gamma_1, \dots, \beta_{g}, \gamma_{g}$ indicated in Figure~\ref{liftedmaps} respectively; for well-definedness, see e.g. \cite[Fact.~3.9 and Prop.~3.11]{FM}. Similarly, we define $\Psi : B_{2n-1} \to \MCGrand(g,1)$ by mapping $\sigma_1, \dots, \sigma_{2n-2}$ to twists about $\gamma_0, \beta_1, \gamma_1, \dots, \gamma_{g-1}, \beta_g$ respectively (Figure~\ref{liftedmapsrand}). Notice also that these homomorphisms are exactly those that appear in the Birman--Hilden theorem \cite{BH} (see also \cite[Sect.~9.4]{FM}).

\begin{figure}[h]
\begin{subfigure}[b]{\textwidth}
\centering
\begin{overpic}[scale=0.75]{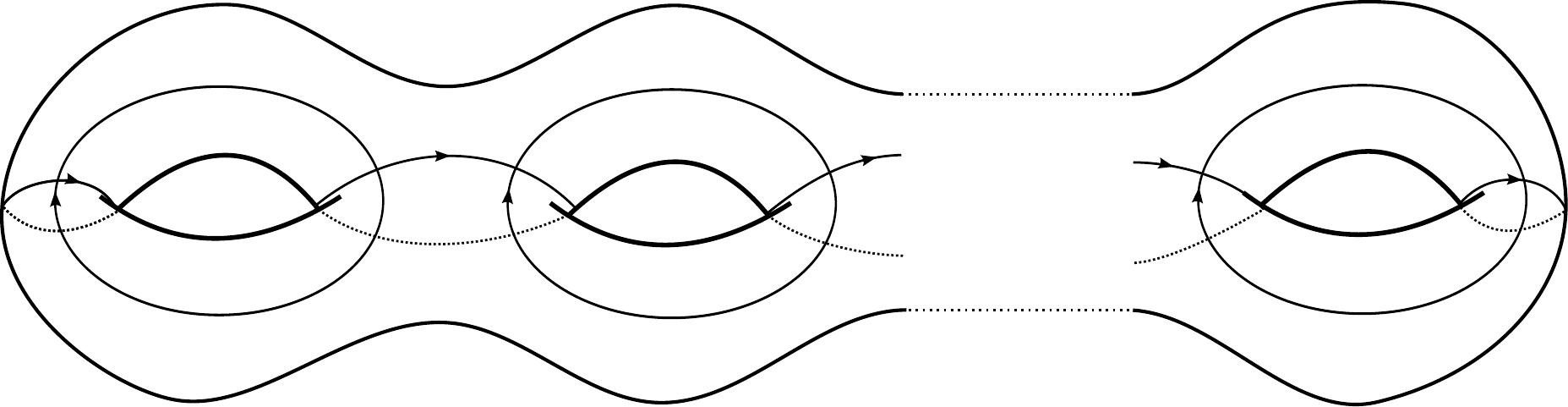}
	\put(52,88){$\beta_1$}
	\put(167,88){$\beta_2$}
	\put(344,88){$\beta_g$}
  \put(-12,50){$\gamma_0$}
	\put(110,70){$\gamma_1$}
	\put(232,70){$\gamma_2$}
	\put(280,70){$\gamma_{g-1}$}
	\put(406,51){$\gamma_{g}$}
\end{overpic}
\caption{The closed surface $\Sigma_g$.}
\label{liftedmaps}
\end{subfigure}
\\[0.5cm]
\begin{subfigure}[b]{\textwidth}
\centering
\begin{overpic}[scale=0.75]{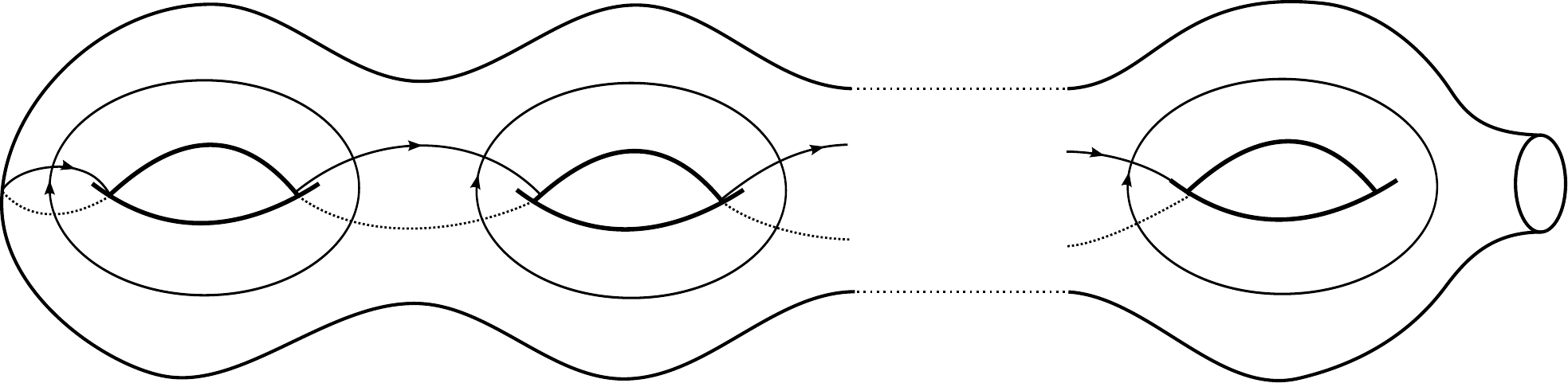}
	\put(52,88){$\beta_1$}
	\put(167,88){$\beta_2$}
	\put(344,88){$\beta_g$}
  \put(-12,50){$\gamma_0$}
	\put(110,70){$\gamma_1$}
	\put(232,70){$\gamma_2$}
	\put(280,70){$\gamma_{g-1}$}
\end{overpic}
\caption{The surface $\Sigma_g^1$.}
\label{liftedmapsrand}
\end{subfigure}
\\[0.5cm]
\begin{subfigure}[b]{\textwidth}
\centering
\begin{overpic}[scale=0.75]{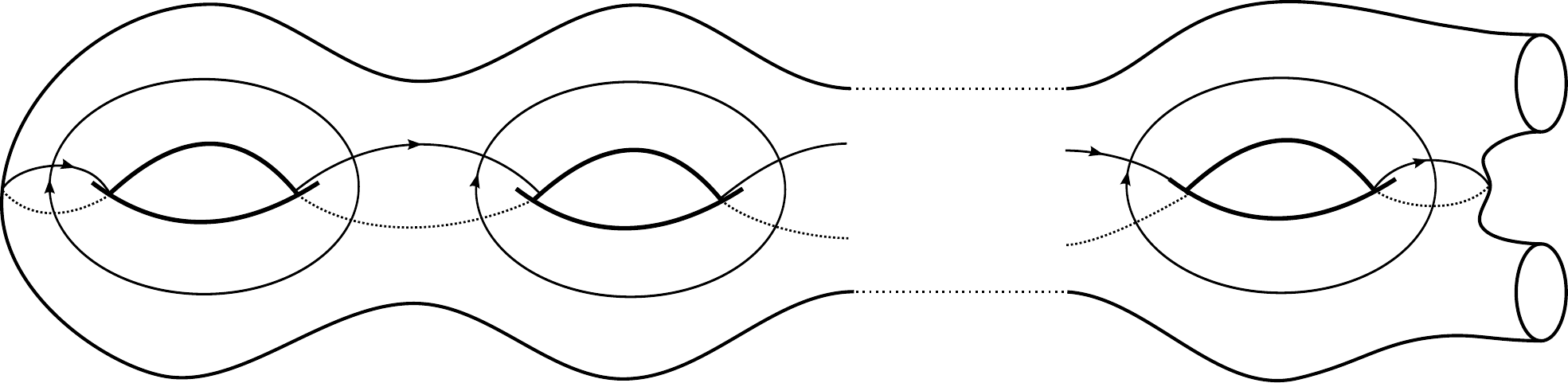}
	\put(52,88){$\beta_1$}
	\put(167,88){$\beta_2$}
	\put(345,88){$\beta_g$}
  \put(-12,50){$\gamma_0$}
	\put(110,70){$\gamma_1$}
	\put(232,70){$\gamma_2$}
	\put(280,70){$\gamma_{g-1}$}
	\put(410,54){$\gamma_{g}$}
\end{overpic}
\caption{The surface $\Sigma_g^2$.}
\label{liftedmapsrande}
\end{subfigure}
\caption{Our naming convention for curves and homology cycles on a surface.}
\label{flader}
\end{figure}
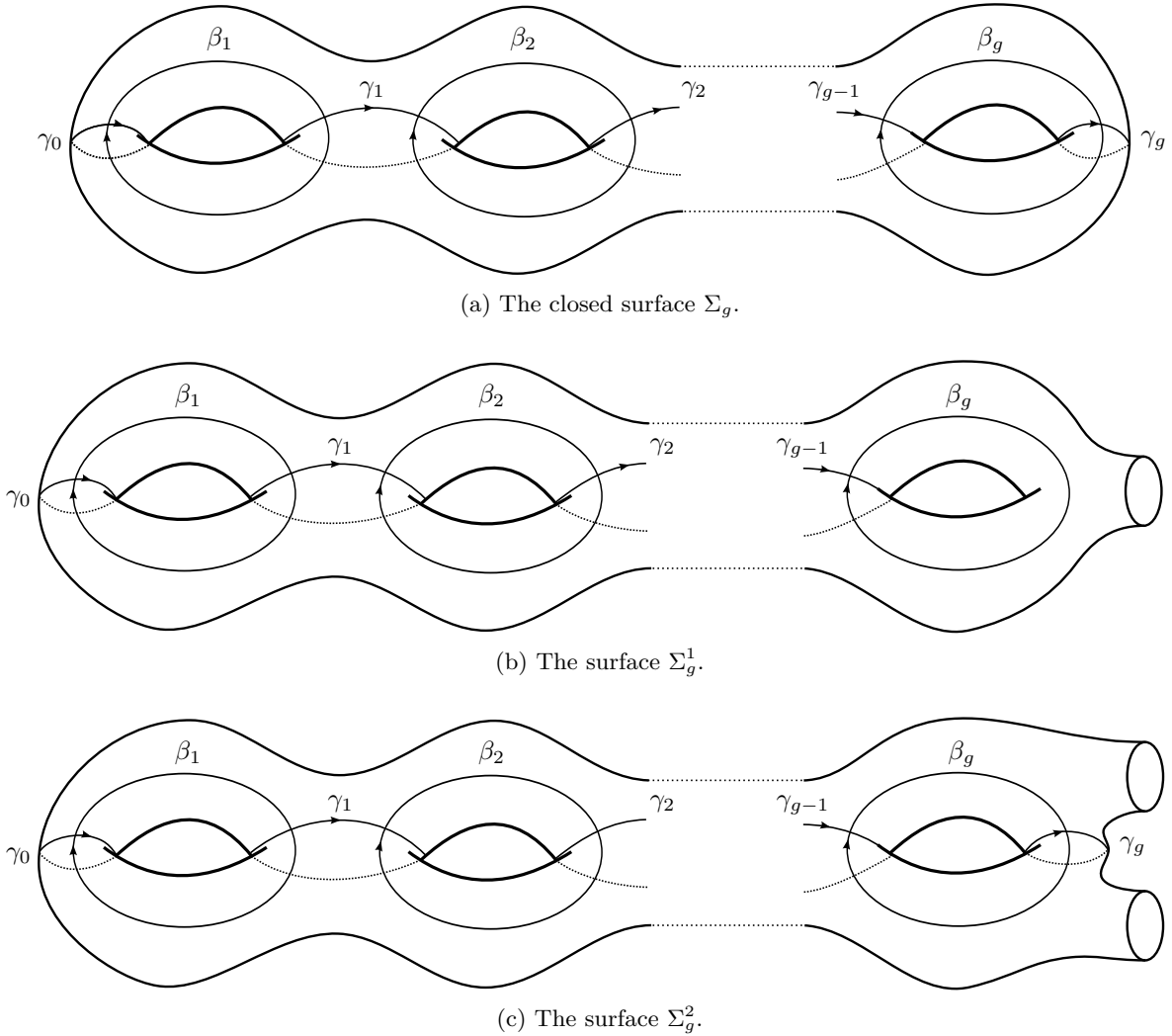

Now, $\MCG(g,0)$ and $\MCGrand(g,1)$ act on the corresponding first homology groups by symplectomorphisms with respect to the intersection pairing $\omega$. For $m = 0,1$ and $l \geq 1$, we let
\begin{align*}
	\tilde{\rho}^{g,l}_{\hom} : \MCGrand(g,m) \to \GL(\Lambda^l H_1(\Sigma_g^m,\bbC) / (\omega \wedge \Lambda^{l-2} H_1(\Sigma_g^m,\bbC)))
\end{align*}
denote the induced action, i.e. for $\phi \in \MCGrand(g,m)$ and $v_1, \dots, v_l \in H_1(\Sigma_g^m,\bbC)$,
\begin{align*}
	\tilde{\rho}_{\hom}^{g,l}(\phi) [v_1 \wedge \dots \wedge v_l] = [(\phi_*)v_1 \wedge \dots \wedge (\phi_*)v_l],
\end{align*}
where we use the conventions that $\Lambda^{-1} H_1(\Sigma_g^m,\bbC) = \{0\}$ and $\Lambda^{0} H_1(\Sigma_g^m,\bbC) = \bbC$. Finally, let $\rho_{\hom}^{g,l} = \tilde{\rho}_{\hom}^{g,l} \circ \Psi$ denote the corresponding braid group representations.

The study of the Jones representation as contained in this paper was initiated by an attempt to generalize the results of \cite{AMU} to general punctured spheres. Here, the authors show how to relate the representations of the mapping class group of a four times punctured sphere, obtained from the level $k$ quantum representations -- closely related to specializations of the parameter $q$ of the Jones representations to certain roots of unity -- to an action on the homology of a torus by considering the limit $k \to \infty$.

In \cite[Sect.~10]{JonHecke}, Jones gave explicit matrices for the representation associated to ${\fontsize{4pt}{12pt}\selectfont {\yng(2,2,2)}}$ (closely related to $\pi_q^{6,0}$, the representation associated to ${\fontsize{4pt}{12pt}\selectfont {\yng(3,3)}}$). Moreover, the choice of basis is such that all matrix entries are in $\bbZ[q,q^{-1}]$, so that one obtains this way representations of $B_6$ for all non-zero values of $q$ rather than only the $6$-regular ones. Kasahara \cite[Lem.~2.1]{Kas} observed in the same vein as above that at $q = -1$, the resulting representation is equivalent to the representation $\rho_{\hom}^{2,2}$. In general, the representation space of $\pi_q^{2n,0}$ has dimension $C_{n+1}$, the $(n+1)$'st Catalan number. Inspired by Kasahara's result and noting the relation
\begin{align*}
  C_{n+1} = \binom{2n}{n}-\binom{2n}{n-2},
\end{align*}
one arrives at the following, which we will prove in Section~\ref{mainthmproof}.
\begin{thm}
	\label{mainthm}
  The Jones representation $\pi_{q}^{2n,0}$ has a natural extension to $q = -1$, for which it is equivalent to $\rho_{\hom}^{g,g}$, where $g = n-1$.
\end{thm}
As it is well-known that $\pi_q^{2n-1,1}$ is equivalent to $\pi_q^{2n,0}|_{B_{2n-1}}$, the above Theorem allows us to deduce a similar homological description for $\pi_q^{2n-1,1}$.

On the other hand, $\pi_q^{2n-1,2n-3}$ is nothing but the reduced Burau representation, which at $q = -1$ is known to be equivalent to the representation $\rho_{\hom}^{g,1}$ on $H_1(\Sigma_g^1,\bbC)$, $g = n-1$. This suggests that the intermediary representations $\pi_q^{2n-1,d}$ interpolate at $q = -1$ between appropriate exterior powers of homological representations, and the following result shows that this is indeed the case.
\begin{thm}\label{thmuligen}
  The representation $\pi^{2n+1,d}_q$ has a natural extension to $q = -1$ for which it is equivalent to the action $\rho_{\hom}^{g,l}$, where $g = n$ and $2l = 2n+1-d$.
\end{thm}
This leaves us with the two-row diagrams having an even number of boxes. The fact that $\pi^{n,d}_q|_{B_{n-1}} \isom \pi^{n-1,d-1}_q \oplus \pi^{n-1,d+1}_q$ (see e.g. \cite{JonHecke}), together with the observation that
\begin{align*}
	\binom{2g}{l} - \binom{2g}{l-2} + \binom{2g}{l+1} - \binom{2g}{l-1} = \binom{2g+1}{l+1} - \binom{2g+1}{l-1}
\end{align*}
for all $g$ and $l$, leads us to the following: let $\hat{\rho}_{\hom}^{g,l}$ denote the action of $B_{2n}$ on $\Lambda^l H_1(\Sigma_{g}^2,\bbC)$, where $g = n-1$, given by mapping $\sigma_1,\dots,\sigma_{2n-1}$ to the action induced by the homological action of the Dehn twists $t_{\gamma_0}, t_{\beta_1}, t_{\gamma_1}, \dots, t_{\beta_{g}},t_{\gamma_g}$ respectively (see Figure~\ref{liftedmapsrande}).
\begin{thm}\label{thmligengenyoung}
  The representation $\pi^{2n,d}_q$ has a natural extension to $q = -1$ for which it is equivalent to a subrepresentation of $\hat{\rho}_{\hom}^{g,l}$, where $g = n-1$ and $2l = 2n-d$.
\end{thm}
In each of the three cases, the appropriate intertwining operator is constructed as follows: a natural basis for the representation spaces of $\pi^{n,d}_q$ is given in terms of non-intersecting paths in the relevant punctured disc, connecting the punctures. Regarding the punctures as marked points, the disk is realized as the quotient of a surface by the order two element rotating the surface by $\pi$ along its horizontal axis in Figure~\ref{flader}, allowing us to realize the surface as a well-understood branched double cover. Lifting the non-intersecting paths through the double cover defines a collection of loops in the covering surface, and by taking an appropriately ordered and scaled wedge product of the homology classes of these loops, we obtain our desired linear map.

As the insightful reader has no doubt pondered, the exact choice of curves in the double cover whose twists we have chosen to lift the $B_n$-action to is of little importance, and for Theorems~\ref{thmuligen} and \ref{thmligengenyoung} any linearly independent set of appropriately intersecting curves would do the job. A similar statement is true for Theorem~\ref{mainthm}; the exact conditions will be made precise in Section~\ref{mainthmproof}.

\subsection{Quantum representations}
\label{quantumrepsintro}
\newcommand{\qMCG}{\widetilde M}
Let $\Lambda_{N,k}$ be the set of Young diagrams of height strictly less than $N$ and length at most $k$. This set comes with an involution $^\dagger$, which assigns to a Young diagram $(\lambda_1 \geq \cdots \geq \lambda_p)$ the result of taking the complement of $\lambda$ in the Young diagram consisting of $N$ rows, each containing $\lambda_1$ boxes, and rotating the result 180 degrees (see \cite[p.~206]{Bla}).

Let $\Sigma$ be a surface and $P$ a finite set of pairs $(p,v)$ where $p\in \Sigma$ and $v\in (T_p\Sigma \setminus \{0\})/ \bbR_+$, and let $\lambda  : P \to \Lambda_{N,k}$ be a labelling of $P$. Let $\qMCG(\Sigma,P,\lambda)$ be the set of isotopy classes of diffeomorphisms of $\Sigma$ preserving $P$ and respecting the labelling. The \emph{level $k$ quantum $SU(N)$-representation} assigns to $(\Sigma, P,\lambda)$ a finite dimensional vector space $V_{N,k}(\Sigma, P, \lambda)$ and a projective representation $\rho_{N,k}$ of $ \qMCG(\Sigma, P, \lambda)$ on $V_{N,k}$.

These quantum representations, arising originally from the mathematical realizations of Witten's 2+1 dimensional topological quantum field theory constructed from Chern--Simons theory \cite{WitJones}, have been studied from a number of different perspectives, including that of the representation theory of quantum groups \cite{RT1}, \cite{RT2}, \cite{Tu}, the skein theory of the Jones polynomial \cite{BHMV1}, \cite{BHMV2}, \cite{Bla}, conformal field theory \cite{TUY}, \cite{UenoBook}, \cite{KanieCFTandBG}, \cite{AU2}, and the geometric quantization of moduli spaces of flat connections \cite{Hit}, \cite{asympgeom}. For a complete description of the equivalence between these constructions, we refer to \cite{AU1}, \cite{AU2}, \cite{AU3} and \cite{AU4}. While useful to have in mind, the present work is independent of this equivalence.

Indeed, for our purposes, it will be convenient to take as our definition of quantum representations those that are obtained from the general procedure, described in \cite[Sect.~IV.5]{Tu}, which defines quantum representations from modular functors, which in turn we apply to the modular functor arising from conformal field theory, constructed in \cite{AU2}. The label set of this theory is a certain $k$-dependent set of dominant integral $\mathfrak{sl}(N,\bbC)$-weights which is in natural correspondence with the set $\Lambda_{N,k}$ defined above (see \cite[Lemma~7.1]{AU4}).

In \cite[Sects.~IV.4--9]{Tu}, the projective ambiguity of the quantum representations is discussed in detail. We record here the fact that if $\Sigma$ has genus 0, the quantum representations define not just projective representations, but actual representations. To be concrete, we consider $\Sigma = S^2 = \bbC \cup \{\infty\}$, and let the set $P$ consist of the points $\{1,\dots,n,\infty\}$, all framed along the real axis. Let $d \in \{0,\dots,n-1\}$ have the same parity as $n$, and label the $n+1$ points by $\boxdiagram,\dots,\boxdiagram,\lambda^\dagger$, where $\lambda = ((n+d)/2 \geq (n-d)/2)$ if $N > 2$ and $\lambda = (d)$ if $N =2$. Assume that $k$ is large enough that $\lambda \in \Lambda_{N,k}$. We write $\qMCG(0,n)_\infty$ for the subgroup of $\qMCG(\Sigma,P,(\boxdiagram,\dots,\boxdiagram,\lambda^\dagger))$ consisting of classes of diffeomorphisms fixing $\infty$ and its framing. Moreover, we introduce $\qMCG(0,n) = \qMCG(\Sigma,\tilde{P},(\boxdiagram,\dots,\boxdiagram))$, where $\tilde{P} = P \setminus \{\infty\}$.

Now, $\qMCG(0,n)_\infty$ is the ribbon braid group on $n$ ribbons, which contains the ordinary braid group $B_n$ as a subgroup, by assigning the blackboard framing to a braid. We will denote by $\rho_{N,k}^\lambda$ the restriction of the quantum representation of $\qMCG(0,n)_\infty$ to this $B_n$, and will think of $B_n$ as $\MCG(0,n)_\infty$, the mapping class group of a sphere with $n$ marked points and an extra marked point called $\infty$ which also carries a preserved projective tangent vector. Finally, we denote by $\rho_{N,k}$ the quantum representation of $\qMCG(0,n)$.

\subsection{The AMU conjecture}
The AMU conjecture \cite[Conjecture~2.4]{AMU} is concerned with the images of pseudo-Anosovs in the quantum $\SU(N)$-representations of mapping class groups for large values of $k$. As mentioned above, in \cite{AMU} the authors prove their conjecture for pseudo-Anosovs in $\MCG(0,4)$ by relating a limit of quantum representation to homology, and it is this argument that we seek to generalize in the present paper. With Theorem~\ref{AMUulige} below, we provide a large family of mapping classes for which the conjecture holds; previously, the conjecture has only been studied for $\MCG(0,4)$ in \cite{AMU}, and for $\MCG(1,1)$ in the closely related quantum $\SO(3)$-representations in unpublished work by Masbaum \cite[Rem.~5.9]{AMU}, \cite{MasOW}, and in \cite{San}.

More precisely, we apply Theorem~\ref{mainthm} to prove the AMU conjecture for those pseudo-Anosov mapping classes whose dynamics are described in homological terms; one particular family for which this is the case are the ones we will refer to as homological pseudo-Anosovs. Here, we say that $\phi \in \MCG(0,n)_\infty$ is a \emph{homological pseudo-Anosov} if its image in $\MCG(0,n+1)$ -- under the map forgetting the tangent vector at $\infty$ -- admits invariant transverse measured singular foliations having only even-pronged non-puncture singularities and only odd-pronged puncture singularities away from $\infty$. Note that we will systematically make the common abuse of notation and think of $\phi$ as both a mapping class and a pseudo-Anosov representing it.

\begin{thm}
	\label{AMUulige}
	Let $\phi\in \MCG(0,n)_\infty$ be a homological pseudo-Anosov, and let $\lambda$ be as in Section~\ref{quantumrepsintro}. Then $\rho_{N,k}^\lambda(\phi)$ has infinite order for all but finitely many $k$.  
\end{thm}
As above, we say that $\phi \in \qMCG(0,n)$ is a \emph{homological pseudo-Anosov} if its image in $\MCG(0,n)$ is a pseudo-Anosov whose invariant foliations have the property that all non-puncture singularities are even-pronged and all punctures have odd-pronged singularities.

\begin{thm}
	\label{AMUlige}
  Assume that $n$ is even. If $\phi \in \qMCG(0,n)$ is a homological pseudo-Anosov, then $\rho_{2,k}(\phi)$ has infinite order for all but finitely many $k$. 
\end{thm}

Finally, let us mention one way our main theorem may be used to deal with general, not necessarily homological, pseudo-Anosovs.

\begin{cor}
	\label{AMUcor}
	Let $f$ be a representative for $[f]\in \MCG(0,n)_\infty$ whose image in $\MCG(0,n+1)$ is pseudo-Anosov, and let $P$ denote the set of singularities of the (un)stable foliation for $f$ with an odd number of prongs. Then $f$ defines an element in $\qMCG(S^2,\abs{P})_\infty$, and the image of $[f]$ under the quantum $\SU(N)$-representation -- where each point of $P$ is labelled by $\boxdiagram$, and $\infty$ is labelled by any $\lambda$ as in Section~\ref{quantumrepsintro} -- has infinite order for all but finitely many values of $k$.
	
	Likewise, if $f$ represents a pseudo-Anosov $[f] \in \MCG(0,n)$ with $P$ the set of odd-pronged singularities, then $f$ defines an element of $\MCG(0,\abs{P})$ whose image under $\rho_{2,k}$ has infinite order for all but finitely many levels.
\end{cor}
Strictly speaking, in order to deal with the issue of framings, we defined only the genus $0$ quantum representations in the case where the marked points were $1,\dots,n$ and possibly $\infty$, so to be precise we should include in the statement a conjugation by a group isomorphism identifying the mapping class groups; this does not change the conclusion.
\begin{proof}[Proof of Corollary~\ref{AMUcor}]
	The first claim follows directly from Theorem~\ref{AMUulige}. By the Euler--Poincar\'e formula, a singular foliation must have an even number of singularities with an odd number of prongs, and so the second claim follows from Theorem~\ref{AMUlige}.
\end{proof}

Moreover, we show in Corollary~\ref{sfcor} that for the homological pseudo-Anosovs, the quantum representations determine their stretch factors, answering positively \cite[Question~1.1~(2)]{AMU} in this case. In Appendix~A we discuss the extent to which this is true for non-homological pseudo-Anosovs. As a different spin-off, we show in Proposition~\ref{inforder8} that there are quantum $\SU(2)$-representations whose image at level $8$ is infinite; this extends an earlier result due to Masbaum, \cite{Masinf}.

Finally, we should note that the related problem of giving homological interpretations of the Burau representation of braids for various values of $q$ in terms of finite covers, and in particular its relation to determining stretch factors of pseudo-Anosovs via representations, has been studied by a number of people; see e.g. \cite{BandBoyland}, \cite{McM}, \cite{Kober} and the references contained in these papers.

\textbf{Acknowledgements.} We would like to thank J{\o}rgen Ellegaard Andersen for first introducing us to the conjectures of \cite{AMU}, as the attempt to generalize the results of \cite{AMU} is what led us to the discoveries in this paper.

\section{Jones representations}

\newcommand{\TL}{\mathrm{TL}}
\begin{figure}
\begin{tikzpicture}[scale=0.25]
\foreach \x in {0,1,...,3}{
\node[draw,circle,inner sep=1pt,fill] at (\x,0) {};
\node[draw,circle,inner sep=1pt,fill] at (\x,4) {};
}
\draw (0,0) to[out=90,in=-90] (2,4);
\draw (1,0) to[out=90,in=-90] (3,4);
\draw (2,0) to[out=90,in=180] (2.5,1);
\draw (2.5,1) to[out=0,in=90] (3,0);
\draw (0,4) to[out=-90,in=-180] (0.5,3);
\draw (0.5,3) to[out=0,in=-90] (1,4);

\foreach \x in {0,1,...,3}{
\node[draw,circle,inner sep=1pt,fill] at (\x+6,0) {};
\node[draw,circle,inner sep=1pt,fill] at (\x+6,4) {};
}
\draw (3+6,0) to[out=90,in=-90] (1+6,4);
\draw (8,0) to[out=90,in=-90] (6,4);
\draw (7,0) to[out=90,in=0] (6.5,1);
\draw (6.5,1) to[out=180,in=90] (6,0);
\draw (9,4) to[out=-90,in=0] (8.5,3);
\draw (8.5,3) to[out=180,in=-90] (8,4);

\foreach \x in {0,1,...,3}{
\node[draw,circle,inner sep=1pt,fill] at (\x+14,2) {};
\node[draw,circle,inner sep=1pt,fill] at (\x+14,6) {};
\node[draw,circle,inner sep=1pt,fill] at (\x+14,-2) {};
}
\draw (3+6+8,-2) to[out=90,in=-90] (1+6+8,2);
\draw (8+8,-2) to[out=90,in=-90] (6+8,2);
\draw (7+8,-2) to[out=90,in=0] (6.5+8,-1);
\draw (6.5+8,-1) to[out=180,in=90] (6+8,-2);
\draw (9+8,2) to[out=-90,in=0] (8.5+8,1);
\draw (8.5+8,1) to[out=180,in=-90] (8+8,2);

\draw (14,2) to[out=90,in=-90] (16,6);
\draw (15,2) to[out=90,in=-90] (17,6);
\draw (16,2) to[out=90,in=180] (16.5,3);
\draw (16.5,3) to[out=0,in=90] (17,2);
\draw (14,6) to[out=-90,in=-180] (14.5,5);
\draw (14.5,5) to[out=0,in=-90] (15,6);

\foreach \x in {0,1,...,3}{
\node[draw,circle,inner sep=1pt,fill] at (\x+28,0) {};
\node[draw,circle,inner sep=1pt,fill] at (\x+28,4) {};
}
\draw (28,4) to[out=-90,in=-180] (28.5,3);
\draw (28.5,3) to[out=0,in=-90] (29,4);
\draw (29,0) to[out=90,in=0] (28.5,1);
\draw (28.5,1) to[out=180,in=90] (28,0);
\draw (30,0) to[out=90,in=-90] (30,4);
\draw (31,0) to[out=90,in=-90] (31,4);
\draw (28,2) node[left] {$=(-A^2-A^{-2})$};
\draw (5.5,2) node[left] {$\circ$};
\draw (12.5,2) node[left] {$=$};

\end{tikzpicture}
\caption{The multiplication in $\TL_4$  \label{fig:TLmult}}
\end{figure}
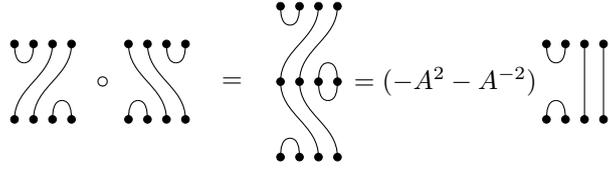

\begin{figure}
\begin{tikzpicture}[scale=0.25]

\foreach \x in {0,1,...,5}{
\node[draw,circle,inner sep=1pt,fill] at (\x-2,0) {};
\node[draw,circle,inner sep=1pt,fill] at (\x-2,4) {};
}
\draw (0,4) to[out=-90,in=-180] (0.5,3);
\draw (0.5,3) to[out=0,in=-90] (1,4);
\draw (1,0) to[out=90,in=0] (0.5,1);
\draw (0.5,1) to[out=180,in=90] (0,0);
\draw (2,0) to[out=90,in=-90] (2,4);
\draw (3,0) to[out=90,in=-90] (3,4);
\draw (-1,0) to[out=90,in=-90] (-1,4);
\draw (-2,0) to[out=90,in=-90] (-2,4);

\end{tikzpicture}
\caption{The element $e_3\in \TL_6$  \label{fig:TLei}}
\end{figure}
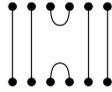

\label{jonesdefsect}
In this section, we briefly recall the definition of the Jones representations of braid groups, fixing the notation and normalizations used throughout the rest of the paper.

The Temperley--Lieb algebra $\TL_n$ is an algebra over $\bbC (A)$ and has a basis consisting of noncrossing pairings of $2n$ points, $n$ of them located at the bottom of a square and $n$ of them at the top. The multiplication is given by stacking two squares on top of each other, rescaling the vertical direction to obtain a square, removing all circles, and multiplying by $-A^2-A^{-2}$ for each removed circle, see Figure~\ref{fig:TLmult}. It is generated, as an algebra, by $n$ elements, $\id,e_1, e_2, \dots, e_{n-1}$ (Figure~\ref{fig:TLei}). Following Jones, we define a representation of the braid group $B_n$ by $\sigma_i \mapsto A\id + A^{-1}e_i$. There exists a Markov trace on $\TL_n$ and the trace of the Jones representation is essentially the Kauffman bracket of the closure of the braid.

For generic $A$, the irreducible sectors of this braid group action are in correspondence with $d \in \{0, \dots, n\}$ with $d \equiv n \mod 2$, and each of the underlying vector spaces has a basis given by trivalent trees with $n+1$ leaves, coloured admissibly by Jones--Wenzl idempotents, such that $n$ leaves are coloured by $1$ and the remaining leaf by $d$. This description is well-known to experts, but for concrete and concise reference, see e.g. \cite[Thm.~1.20,~Thm.~1.23]{Wan}; indeed the treatment contained in \cite[Sect.~1]{Wan} completely covers what we need.

It is possible to specialize the value of $A$ to a non-zero complex number, obtaining the $\bbC$-algebra $\TL_n(A)$, and as long as $A$ is not a $4k+8$'th root of unity for any $k < n$, this algebra is semi-simple, and the above discussion provides us with representations $\rho_{A}^{n,d} : B_n \to W_{A}^{n,d}$ (we remark that in the notation of \cite[Sect.~1.1.6]{Wan}, $W_A^{n,d} = \Hom(d,1^n)$, and the basis elements described above are denoted $\{e_{C,d}\}$).

\begin{lem}
  Let $A \in \bbC$ be as above, and let $q = A^{4}$. Then the braid group representation $\sigma_i \mapsto A^{-1} \rho_A^{n,d}(\sigma_i)$ is equivalent to $\pi_q^{n,d}$.
\end{lem}
\begin{proof}
  Once again, this is well-known to experts. A basis of the representation space for $\pi_q^{n,d}$ is given by standard Young tableaux on the relevant Young diagram, to each such tableau there is naturally associated a coloured tree in $W^{n,d}_A$, and this association is bijective. Thus, showing the equivalence boils down to checking that the actions on each of the bases agree under this correspondence. For $\pi_q^{n,d}$, this action is given explicitly in \cite[(2.3)]{Wen}, and for $\rho_A^{n,d}$ as explicitly on \cite[p.~18]{Wan}.
\end{proof}
Recall that our goal is to manipulate the $\pi_q^{n,d}$ into representations that we may specialize to all non-zero values of $q$, rather than just the $n$-regular ones. With an eye towards Wenzl's recursive formula for the Jones--Wenzl idempotents, one finds that all problematic factors occur in the terms that are not of ``highest order'' (in the language of \cite{FWW}). This leads us to the following:

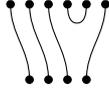
\begin{figure}
\begin{tikzpicture}[scale=0.25]

\foreach \x in {0,1,...,5}{
\node[draw,circle,inner sep=1pt,fill] at (\x-2,4) {};
}
\foreach \x in {0,1,...,3}{
\node[draw,circle,inner sep=1pt,fill] at (\x-1,0) {};
}
\draw (1,4) to[out=-90,in=-180] (1.5,3);
\draw (1.5,3) to[out=0,in=-90] (2,4);
\draw (2,0) to[out=90,in=-90] (3,4);
\draw (1,0) to[out=90,in=-90] (0,4);
\draw (0,0) to[out=90,in=-90] (-1,4);
\draw (-1,0) to[out=90,in=-90] (-2,4);

\end{tikzpicture}
\caption{A basis element of $V^{6,4}$.}
\label{ekspaabasiselement}
\end{figure}

\begin{figure}
\begin{tikzpicture}[scale=0.25]
\foreach \x in {0,1,...,5}{
\node[draw,circle,inner sep=1pt,fill] at (\x-2,4) {};
}
\foreach \x in {0,1,...,3}{
\node[draw,circle,inner sep=1pt,fill] at (\x-1,0) {};
}
\draw (-2,2) node[left] {$\rho_A(\sigma_2)$};
\draw (1,4) to[out=-90,in=-180] (1.5,3);
\draw (1.5,3) to[out=0,in=-90] (2,4);
\draw (2,0) to[out=90,in=-90] (3,4);
\draw (1,0) to[out=90,in=-90] (0,4);
\draw (0,0) to[out=90,in=-90] (-1,4);
\draw (-1,0) to[out=90,in=-90] (-2,4);
\draw (6.5,2) node[left] {$=A$};

\foreach \x in {0,1,...,5}{
\node[draw,circle,inner sep=1pt,fill] at (\x-2+8.5,4) {};
}
\foreach \x in {0,1,...,3}{
\node[draw,circle,inner sep=1pt,fill] at (\x-1+8.5,0) {};
}
\draw (1+8.5,4) to[out=-90,in=-180] (1.5+8.5,3);
\draw (1.5+8.5,3) to[out=0,in=-90] (2+8.5,4);
\draw (2+8.5,0) to[out=90,in=-90] (3+8.5,4);
\draw (1+8.5,0) to[out=90,in=-90] (0+8.5,4);
\draw (0+8.5,0) to[out=90,in=-90] (-1+8.5,4);
\draw (-1+8.5,0) to[out=90,in=-90] (-2+8.5,4);
\end{tikzpicture}
\\[0.2cm]
\begin{tikzpicture}[scale=0.25]
\foreach \x in {0,1,...,5}{
\node[draw,circle,inner sep=1pt,fill] at (\x-2,4) {};
}
\foreach \x in {0,1,...,3}{
\node[draw,circle,inner sep=1pt,fill] at (\x-1,0) {};
}
\draw (-2,2) node[left] {$\rho_A(\sigma_3)$};
\draw (1,4) to[out=-90,in=-180] (1.5,3);
\draw (1.5,3) to[out=0,in=-90] (2,4);
\draw (2,0) to[out=90,in=-90] (3,4);
\draw (1,0) to[out=90,in=-90] (0,4);
\draw (0,0) to[out=90,in=-90] (-1,4);
\draw (-1,0) to[out=90,in=-90] (-2,4);
\draw (6.5,2) node[left] {$=A$};

\foreach \x in {0,1,...,5}{
\node[draw,circle,inner sep=1pt,fill] at (\x-2+8.5,4) {};
}
\foreach \x in {0,1,...,3}{
\node[draw,circle,inner sep=1pt,fill] at (\x-1+8.5,0) {};
}
\draw (1+8.5,4) to[out=-90,in=-180] (1.5+8.5,3);
\draw (1.5+8.5,3) to[out=0,in=-90] (2+8.5,4);
\draw (2+8.5,0) to[out=90,in=-90] (3+8.5,4);
\draw (1+8.5,0) to[out=90,in=-90] (0+8.5,4);
\draw (0+8.5,0) to[out=90,in=-90] (-1+8.5,4);
\draw (-1+8.5,0) to[out=90,in=-90] (-2+8.5,4);
\draw (16.5,2) node[left] {$+A^{-1}$};
\foreach \x in {0,1,...,5}{
\node[draw,circle,inner sep=1pt,fill] at (\x-2+18.5,4) {};
}
\foreach \x in {0,1,...,3}{
\node[draw,circle,inner sep=1pt,fill] at (\x-1+18.5,0) {};
}
\draw (0+18.5,4) to[out=-90,in=-180] (0.5+18.5,3);
\draw (0.5+18.5,3) to[out=0,in=-90] (1+18.5,4);
\draw (2+18.5,0) to[out=90,in=-90] (3+18.5,4);
\draw (1+18.5,0) to[out=90,in=-90] (2+18.5,4);
\draw (0+18.5,0) to[out=90,in=-90] (-1+18.5,4);
\draw (-1+18.5,0) to[out=90,in=-90] (-2+18.5,4);
\end{tikzpicture}
\\[0.2cm]
\begin{tikzpicture}[scale=0.25]
\foreach \x in {0,1,...,5}{
\node[draw,circle,inner sep=1pt,fill] at (\x-2,4) {};
}
\foreach \x in {0,1,...,3}{
\node[draw,circle,inner sep=1pt,fill] at (\x-1,0) {};
}
\draw (-2,2) node[left] {$\rho_A(\sigma_4)$};
\draw (1,4) to[out=-90,in=-180] (1.5,3);
\draw (1.5,3) to[out=0,in=-90] (2,4);
\draw (2,0) to[out=90,in=-90] (3,4);
\draw (1,0) to[out=90,in=-90] (0,4);
\draw (0,0) to[out=90,in=-90] (-1,4);
\draw (-1,0) to[out=90,in=-90] (-2,4);
\draw (9.5,2) node[left] {$=-A^{-3}$};

\foreach \x in {0,1,...,5}{
\node[draw,circle,inner sep=1pt,fill] at (\x-2+11.5,4) {};
}
\foreach \x in {0,1,...,3}{
\node[draw,circle,inner sep=1pt,fill] at (\x-1+11.5,0) {};
}
\draw (1+11.5,4) to[out=-90,in=-180] (1.5+11.5,3);
\draw (1.5+11.5,3) to[out=0,in=-90] (2+11.5,4);
\draw (2+11.5,0) to[out=90,in=-90] (3+11.5,4);
\draw (1+11.5,0) to[out=90,in=-90] (0+11.5,4);
\draw (0+11.5,0) to[out=90,in=-90] (-1+11.5,4);
\draw (-1+11.5,0) to[out=90,in=-90] (-2+11.5,4);
\end{tikzpicture}
\caption{The actions of $\sigma_2$, $\sigma_3$, and $\sigma_4$ on the basis element from Figure~\ref{ekspaabasiselement}.}
\label{ekspaavirkning}
\end{figure}
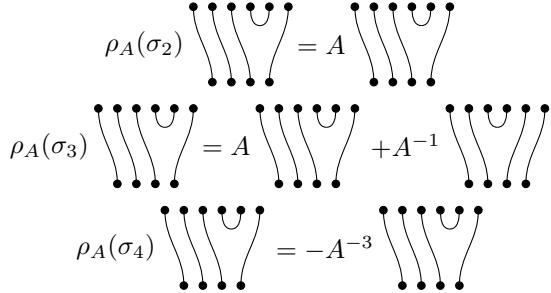

\begin{definition}
  Let $V^{n,d}$, $0 \leq d \leq n$, denote the complex vector space spanned by non-crossing pairings of $n+d$ points, such that $n$ points are ``at the top'' and $d$ points ``at the bottom'', in such a way that a point at the bottom will be connected to a point at the top; see Figure~\ref{ekspaabasiselement}.
  
  The braid group $B_n$ acts on this space exactly as on the Temperley--Lieb algebra: by stacking braids on diagrams, and resolving the result by using the Kauffman bracket. However, if two bottom points end up connected when doing so, the result is multiplied by $0$; see Figure~\ref{ekspaavirkning}.
\end{definition}
\begin{lem}
	The representation $\diagramrep_{A}^{n,d}$ of $B_n$ on $V^{n,d}$ is equivalent to $\rho_A^{n,d}$ when $A \in \bbC \setminus \{0\}$ is not a $4k+8$'th root of unity for any $k < n$.
\end{lem}
\begin{proof}
  Recall that $W_A^{n,d}$ has a basis given in terms of trees coloured by Jones--Wenzl idempotents. Replacing an idempotent by the corresponding multicurve, ignoring lower order terms, we obtain a map $W_A^{n,d} \to V^{n,d}$, which is $B_n$-equivariant by construction. By assumption on $A$, the $k$-admissibility condition of coloured trees (c.f. \cite[Def.~1.24]{Wan}) boils down to the triangle inequality, which in turn means that the map is surjective.
\end{proof}

Since $\diagramrep_A^{n,d}$ is defined over $\bbZ[A,A^{-1}]$, we obtain our desired extension of $\pi_q^{n,d}$ by letting $q = A^4$ and taking the tensor product with the $1$-dimensional representation $\sigma_i \mapsto A$.

We should note that considering the representation in this particular basis is not a new idea: see for instance \cite{RSA} for a thorough description of the action which moreover puts focus on the case $q = -1$ which is of interest to us.

\section{Proofs of main theorems}
\label{mainthmproof}
In this section, we show Theorems~\ref{mainthm}, \ref{thmuligen}, and \ref{thmligengenyoung}. Throughout the section, we assume that $A = \exp(-\pi i/4)$. 

\subsection{Construction of the (iso)morphism}
We now construct the intertwining morphism $\phi$ as needed in Theorems~\ref{thmuligen} and \ref{thmligengenyoung}. For Theorem~\ref{mainthm} the construction needs to be tweaked slightly; a concern we defer until it becomes relevant.

For a diagram $D$ in the basis of $V^{n,d}$ we denote by $D_0$ the set of arcs connecting two top points. We identify the points at the top with the numbers $1, 2, \dots, n$, and for $a\in D_0$ we denote by $a_0$ and $a_1$ the left and right end point of $a$ respectively. The set of bottom points will be called $\infty$. The symbol $i$ will be used to denote indices and the imaginary unit interchangeably; this should cause no confusion.

\begin{definition}
\label{homomorfidef}
Let $c_1, \dots, c_{n-1}$ be simple closed curves on the surface $\Sigma_g^m$, such that $c_i \cap c_{i+1}$ consists of a single point, and such that $c_i$ and $c_j$ are disjoint for $\abs {i-j}>1$. We orient the curves such that $t_{c_i}(c_{i+1}) = c_{i+1} + c_i$, where we abuse notation and use $c_i$ also for the element that the oriented curve defines in $H_1(\Sigma_g^m, \bbC)$. Assume that $\spn\{c_i \mid i=1, \dots, n-1\}$ has dimension $n-1$. 

We order the arcs in $D_0$ by their starting points so that $e < e'$ if $e_0 < e'_0$. To each arc $e$ in $D_0$ we associate the element $X_e = \sum_{i=e_0}^{e_1-1} c_i \in \spn\{c_i\}\subseteq H_1(\Sigma_g^m,\bbC)$.

Define a map $\phi : V^{n,d} \to \Lambda^l H_1(\Sigma_g^m,\bbC)$, by letting, for a basis diagram $D \in V^{n,d}$,
\begin{align}
  \label{defafmorfi}
	\varphi(D) = f(D) \bigwedge_{e\in D_0} X_e,
\end{align}
where $f(D) = (-i)^{\sum_{e\in D_0} w(e) + v(e)}$; here, $w(e)$ denotes the number of arcs between the starting and ending point of $e$, $v(e)$ denotes the number of points greater than $e_1$ that are connected to $\infty$ (see Figure~\ref{fig:illafwava}), and the wedge product runs through $D_0$ from the first to the final arc with respect to the ordering on $D_0$.

\begin{figure}
\begin{tikzpicture}[scale=0.25]

\newcount\ga
\ga=0
\foreach \x in {6,2,1,3,0,5,4}{
\draw (4*\the\ga,6) to[out=-90,in=-90] (2+4*\x,6);
\global\advance\ga by1
}

\draw[red,thick] (12,6) to[out=-90,in=-90] (14,6);
\draw[red,thick] (4,6) to[out=-90,in=-90] (10,6);
\draw[red,thick] (6,6) to[out=-90,in=-90] (8,6);
\draw (-2,6) -- +(0,-8);
\draw[blue,thick] (28,6) -- +(0,-8);
\draw[blue,thick] (30,6) -- +(0,-8);
\draw (10,3) node[right] {$e$};
\foreach \x in {-1,0,1,...,15}{
\node[draw,circle,inner sep=1pt,fill] at (2*\x,6) {};
}
\end{tikzpicture}
\caption{The arcs contributing to $w(e)=3$ have been coloured red, and the arcs contributing to $v(e)=2$ have been coloured blue. \label{fig:illafwava}}
\end{figure}
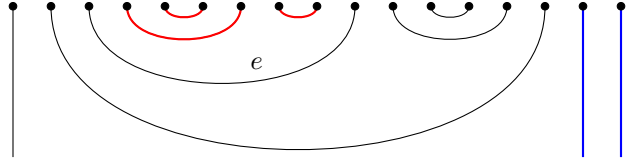

\end{definition}
\subsection{Proof of equivariance}\label{equivariancesubsect}
\begin{lemma}
\label{homomofilemma}
Denote by $T_i$ the action of $e_i$ on $V^{n,d}$. Then we have, for all $i=1, 2, \dots, n-1$ and $D\in V^{n,d}$, that
\[
	\varphi(D+iT_i D) = t_{c_i} \varphi(D).
\]
\end{lemma}
\begin{proof}
Let $c=c_i$. We first observe that $(t_c)_*$ acts trivally on $X_e$ if $e$ is not connected to exactly one of $i$ or $i+1$. In particular, if the points $i$, $i+1$ are connected, or if both are connected to $\infty$, then 
\[
	(t_c)_*( \varphi(D)) = \varphi(D).
\]
Since in these cases $T_i D = 0$, we obtain the claim of the Lemma.

Let us therefore consider the case where we have two distinct arcs connecting to $i$, $i+1$, not both connecting to $\infty$. Let us first assume that $i+1$ is connected to $\infty$, and let us denote the arc ending in $i$ by $a$.

Then $(t_c)_*$ acts trivially on all factors of $\varphi(D)$ except for $X_a$. If $a$ is the $j$'th arc in $D_0$, then
\[
	(t_c)_* \varphi(D) = f(D) X_1^{j-1} \wedge (X_a - c) \wedge X_{j+1}^k= \varphi(D) +  f(D) X_1^{j-1} \wedge (-c) \wedge X_{j+1}^l,
\]
where we denote the $i$'th edge in $D_0$ by $e^i$ and write
\[
	X_i^j = \bigwedge_{l=i}^j X_{e^i}.
\]
On the other hand, we have
\begin{align*}
	\varphi(T_i (D)) &= (-i)^{-2w(a)-1} f(D) X_1^{j-1} \wedge X_{j+1}^{j+w(a)} \wedge c \wedge X_{j+w(a)+1}^l \\
								   &= (-i)^{-2w(a)-1} (-1)^{w(a)} f(D) X_1^{j-1} \wedge c \wedge X_{j+1}^{l} \\
								   &= -if(D) X_1^{j-1} \wedge (-c) \wedge X_{j+1}^{l}.
\end{align*}
Likewise, if $i$ is connected to $\infty$ but $i+1$ is not, we denote by $a$ the arc connecting to $i+1$ and find that
\[
	(t_c)_* \varphi(D) = f(D) X_1^{j-1} \wedge (X_a + c) \wedge X_{j+1}^l= \varphi(D) +  f(D) X_1^{j-1} \wedge  c \wedge X_{j+1}^l,
\]
and
\begin{align*}
	\varphi(T_i(D)) = (-i)^{w(a)-w(a)+1} f(D) X_1^{j-1} \wedge c \wedge X_{j+1}^l 	= -i f(D )X_1^{j-1} \wedge c \wedge X_{j+1}^l .
\end{align*}
Now only the cases where two distinct arcs from $D_0$ connect to $i$ and $i+1$ remain. Denote these arcs by $a$ and $b$, and assume that they are the $j$'th and $m$'th arc respectively, with $j<m$. There are three possibilities, corresponding to whether or not $a$ and $b$ start or end in $i$ and $i+1$ (the case where $a$ starts in $i$ and $b$ ends in $i+1$ is not possible if $a\neq b$). We denote these cases by the signs $(\xi^i_a, \xi^i_b)$ of the intersections of the curves with $c$, i.e. $(t_{c})_* X_a =X_a + \xi^i_a c$, and likewise for $b$. Then
\begin{align*}
	(t_c)_* \varphi(D) &= f(D) X_1^{j-1} \wedge (X_a + \xi_a^i c) \wedge X_{j+1}^{m-1} \wedge (X_b + \xi_b^i c) \wedge X_{m+1}^l \\
			    &= \varphi(D) + f(D) X_1^{j-1} \wedge (\xi^i_b X_a - \xi_a^i X_b)\wedge X_{j+1}^{m-1} \wedge c \wedge X_{m+1}^l \\
			    &= \varphi(D) + f(D) X_1^{j-1} \wedge (\xi^i_b X_a \pm c - \xi_a^i X_b)\wedge X_{j+1}^{m-1} \wedge c \wedge X_{m+1}^l.
\end{align*}
If now $(\xi_a,\xi_b) = (-,+)$, then we have
\begin{align*}
	\varphi(T_i (D))&= (-i)^{-w(a) -w(b) + w(a)+w(b) +1} f(D) X_{1}^{j-1} \wedge (a+c+b) \wedge X_{j+1}^{m-1} \wedge c \wedge X_{m+1}^l \\
					     &= -i f(D) X_{1}^{j-1} \wedge (a+c+b) \wedge X_{j+1}^{m-1} \wedge c \wedge X_{m+1}^l.
\end{align*}
If $(\xi_a,\xi_b) = (+,+)$,
\begin{align*}
	\varphi(T_i (D))&= (-i)^{-w(a) -w(b) + w(a)-w(b) -1} f(D) \\
	                         &\phantom{{}= }\cdot X_{1}^{j-1} \wedge c\wedge X_{j+1}^{m-1} \wedge X_{m+1}^{m+w(b)} \wedge  (a-c-b)  \wedge X_{m+w(b) +1}^l \\
										       &= i(-i)^{-2w(b)}(-1)^{w(b)}  f(D) X_{1}^{j-1} \wedge c \wedge X_{j+1}^{m-1} \wedge (a-c-b) \wedge X_{m+1}^l\\
										       &= -i f(D) X_{1}^{j-1} \wedge (a-c-b) \wedge X_{j+1}^{m-1} \wedge c \wedge X_{m+1}^l.
\end{align*}
And finally, if $(\xi_a,\xi_b) = (-,-)$,
\begin{align*}
	\varphi(T_i (D))&= (-i)^{-w(a) -w(b) + w(a)-w(b) -1} f(D) \\
	                         &\phantom{{}= }\cdot X_{1}^{j-1} \wedge (a-c-b) \wedge X_{j+1}^{m-1} \wedge X_{m+1}^{m+w(b)} \wedge  c  \wedge X_{m+w(b) +1}^l \\
										       &= i(-i)^{-2w(b)}(-1)^{w(b)}  f(D) X_{1}^{j-1} \wedge c \wedge X_{j+1}^{m-1} \wedge (a-c-b) \wedge X_{m+1}^l\\
										       &= -i f(D) X_{1}^{j-1} \wedge (-a+c+b) \wedge X_{j+1}^{m-1} \wedge c \wedge X_{m+1}^l.
\end{align*}
\end{proof}
This shows that $\varphi$ is a homomorphism between the representation $\sigma_i \mapsto A^{-1} \diagramrep^{n,d}_A(\sigma_i)$ and the representation on $\Lambda^l H_1(\Sigma_g^m,\bbC)$ in the cases of Theorems~\ref{thmuligen}, and \ref{thmligengenyoung}. For Theorem~\ref{mainthm}, we define $\tilde{\varphi}(D)$ for $D \in V^{2n,0}$ in the following way: there is an $i$ such that there is an arc $a$ connecting $i$ to $n$. We define an element $\tilde D$ in $V^{n-1,1}$ by removing $a$, forgetting $n$, and connecting $i$ to $\infty$, and we define $\tilde{\phi} : V^{2n,0} \to \Lambda^g H_1(\Sigma_g,\bbC)$ by $\tilde{\varphi}(D) = \varphi(\tilde D)$, identifying $H_1(\Sigma_g,\bbC)$ with $H_1(\Sigma_g^1,\bbC)$. It is clear that the induced map $\tilde{\cdot} : V^{2n,0} \to V^{2n-1,1}$ is an isomorphism of vector spaces.
\begin{lem}
  \label{homomofilemmamainthm}
  The map $\tilde{\phi}$ is a homomorphism of representations.
\end{lem}
In order to prove this, we need the following Lemma.

\begin{figure}
\begin{tikzpicture}[scale=0.25]

\newcount\ga
\ga=0
\foreach \x in {6,2,1,0,3,5,4}{
\draw (4*\the\ga,6) to[out=-90,in=-90] (2+4*\x,6);
\global\advance\ga by1
}

\draw[red,thick] (2,6) to[out=-90,in=-90] (12,6);
\draw[red,thick] (14,6) to[out=-90,in=-90] (16,6);
\draw[red,thick] (18,6) to[out=-90,in=-90] (24,6);

\foreach \x in {0,1,...,13}{
\node[draw,circle,inner sep=1pt,fill] at (2*\x,6) {};
}
\end{tikzpicture}
\caption{Here, the outermost arc is $a$ and the $b_k$'s have been coloured red. \label{fig:lemmatilhom}}
\end{figure}
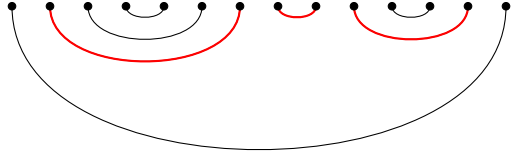

\begin{lemma}
\label{lemmatilhomomorfibevis}
If $D$ is a diagram and $a$ is an arc that connects $i$ to $j$, $i<j$, we have that
\[
	X_a \wedge \bigwedge _{e\prec a} X_e = \left( \sum_{\substack{i\leq k \leq j,\\k \equiv i \mod
2}} c_k  \right) \wedge \bigwedge _{e\prec a} X_e,
\]
where we say that $e\prec a$ if $a_0 < e_0 < e_1 < a_1$. 
\end{lemma}
\begin{proof}
	The proof is by induction on $w(a)$. If $w(a)=0$, the claim is clear. 
Otherwise we use the induction hypothesis on the arcs $b_k$, $k=1,2,\dots, 
m$, connecting $(i_k,j_k)$, where $i_1=i+1$, $j_m = j-1$, and $j_k+1 = i_{k+1}$ for $k<m$. See Figure~\ref{fig:lemmatilhom}.
\end{proof}

Up until this point, we have not used the possible description of the curves $c_i$ in terms of the curves $\beta_i$ and $\gamma_i$ of Figure~\ref{flader}, but for the proof of Lemma~\ref{homomofilemmamainthm} below, we will need that
\[
  c_{n-1} = - \sum_{\substack{1\leq k \leq n-3,\\k \equiv 1 \mod 2}} c_k
\]
in $H_1(\Sigma_g,\bbC)$, which will be guaranteed by considering the concrete curves.
\begin{proof}[Proof of Lemma~\ref{homomofilemmamainthm}]
If we consider only the action of $B_{n-1} \subseteq B_n$, then $\tilde{\phi}$ defines a homomorphism by the above, so we only have to check the equivariance of the action of $\sigma_{n-1}$. If $n-1$ and $n$ are connected, then $\beta_g$ does not appear in $\varphi(\tilde D)$, and the action of $\sigma_{n-1}$ on homology is trivial, just as $T_n(D) = 0$.

If $n$ is connected to $i<n-1$ by the $j$'th arc, and $n-1$ is an end point of the $m$'th arc, which we denote by $b$, then
\begin{align*}
	(t_{c_{n-1}})_* (\varphi(\tilde D)) &= f(\tilde D) X_1^{j-1} \wedge X_{j+1}^{m-1} \wedge(X_b - c_{n-1})\wedge  X_{m+1}^l \\ 
																								      & = \varphi(\tilde D) - f(\tilde D)X_1^{j-1} \wedge X_{j+1}^{m-1} \wedge c_{n-1} \wedge X_{m+1}^l \\
														  &= \varphi(\tilde D) + f(\tilde D)X_1^{j-1} \wedge X_{j+1}^{m-1} \wedge \left( \sum_{\substack{1\leq k \leq n-3,\\k \equiv 1 \mod 
2}} c_k  \right)  \wedge X_{m+1}^l \\
&= \varphi(\tilde D) + f(\tilde D)X_1^{j-1} \wedge X_{j+1}^{m-1} \wedge \left( \sum_{\substack{ i \leq k \leq b_0,\\k \equiv 1 \mod 
2}} c_k  \right)  \wedge X_{m+1}^l.
\end{align*}

Here we used Lemma~\ref{lemmatilhomomorfibevis} in last equality, applying it to the arcs left of $i$ that are not contained between the end points of any other arc, and on the arcs between the end points of $b$. Letting $p$ be the number of arcs to the right of $i$ in $\tilde D$, we find, where we denote by $d$ an arc connecting $i$ with $b_0$, that
\begin{align*}
	\varphi( \widetilde {T_{n-1}(D)} ) &= (-i)^{p-w(b)+p-w(b)-1} f(\tilde D) X_{1}^{j-1} \wedge  X_d \wedge X_{j+1}^{m-1} \wedge X_{m+1}^l  \\
 &= (-i)^{p-w(b)+p-w(b)-1} f(\tilde D)  X_1^{j-1} \wedge \left( \sum_{\substack{ i\leq k \leq b_0,\\k \equiv 1 \mod 
2}} c_k  \right)  \wedge X_{j+1}^{m-1} \wedge X_{m+1}^l \\
&= (-1)^{p-w(b)+1} (-i)^{p-w(b)+p-w(b)-1} f(\tilde D)  X_1^{j-1} \wedge X_{j+1}^{m-1} \wedge \left( \sum_{\substack{ i\leq k \leq b_0,\\k \equiv 1 \mod 
2}} c_k  \right)  \wedge X_{m+1}^l \\
&= -i f(\tilde D)  X_1^{j-1} \wedge X_{j+1}^{m-1} \wedge \left( \sum_{\substack{ i\leq k \leq b_0,\\k \equiv 1 \mod 
2}} c_k  \right)  \wedge X_{m+1}^l. 
\end{align*}
This shows the equivariance of $\tilde{\phi}$.
\end{proof}

\subsection{Injectivity}
In the case of Theorem~\ref{thmuligen} we know that the dimension of $V^{n,d}$ agrees with the dimension of $\Lambda^l H_1(\Sigma_g^1,\bbC) / \omega \wedge \Lambda^{l-2} H_1(\Sigma_g^1,\bbC)$, so we can show that $\varphi$ -- which here denotes the map to the quotient -- is an isomorphism by showing that it is surjective, which we will do next. This also shows the injectivity claimed in Theorem~\ref{mainthm}, by construction of $\tilde{\phi}$.

\begin{proof}[Proof of Theorems~\ref{mainthm} and \ref{thmuligen}]
	Denote by $\bar \varphi$ the composition of $\varphi$ with the projection to the quotient $\Lambda^l H_1(\Sigma_g^1,\bbC)/\omega\wedge \Lambda^{l-2} H_1(\Sigma_g^1, \bbC)$. We just need to show that $\bar \varphi$ is injective, as we know that the spaces have the same dimension. Assume that $v$ is in the kernel of $\bar \varphi$; then 
	\[
		e_i v = T_i v = -i( v+  i T_i v) + iv,
	\]
	so
	\[
		\bar\varphi(e_i v) = -i t_{c_i} \bar\varphi(v) + i\bar \varphi(v) = 0,
	\]
	which shows that $\ker \bar \varphi$ is a $\TL_{n}(\exp(-\pi i/4))$-subrepresentation of $V^{n,d}$. But it follows from Corollary~4.8 of \cite{RSA} that the representation on $V^{n,d}$ is irreducible (as noted in the remarks on the case $\beta = 0$ following their Corollary), so $\ker \bar \varphi$ is either $0$ or $V^{n,d}$. It is easy to check that the latter is not the case, and so we obtain the Theorem.
\end{proof}

\begin{remark}
We remark that the injectivity above can also be shown by explicitly constructing a basis 
$a_1, a_2, \dots , a_i, b_1, b_2, \dots, b_j$ of $\Lambda^l H_1(\Sigma,\bbC)$ in such a way that $\varphi$ surjects onto $\spn\{a_k\}$, $\omega\wedge \Lambda^{l-2} H_1(\Sigma, \bbC) = \spn \{b_k\}$, and $\spn\{a_k\}$ has the right dimension. 
\end{remark}

\begin{proof}[Proof of Theorem~\ref{thmligengenyoung}]
	This follows from Theorem~\ref{thmuligen} in the following way: we can define a map $h_1 :  V^{n,d} \to V^{n-1,d-1}$ that maps $D$ to $0$ if $n$ is not connected to $\infty$, and otherwise to the diagram given by removing the arc going from $n$ to $\infty$; here we let $V^{n-1,-1} = V^{n-1,n+1} = \{0\}$. Likewise, we define a map $h_2 : V^{n,d} \to V^{n-1,d+1}$ that maps a diagram to $0$ if $n$ is connected to $\infty$, and otherwise to the diagram where the point connected to $n$ is now connected to $\infty$. For a diagram of the latter type, we define $g(D) = \alpha X_e$, where $\alpha$ is a scalar depending on $D$, and $e$ is an arc such that $\varphi(D) =g(D)\wedge  \varphi(h_2(D))$. It is clear that $g(D) = \alpha \sum _{i=k}^{n-1} c_i$ for some $k$, and by looking at the block decomposition of $\varphi$ with respect to these two kinds of diagrams and with respect to the two subspaces $c_{n-1}\wedge \Lambda^{l-1} \spn\{c_i \mid i=1, 2, \dots, n-2 \}$ and its orthogonal complement (using wedge products of $c_i$'s as an orthogonal basis), we see that $\varphi$ is injective, as it has the form
	\[
	\left(\!\begin{array}{c|c}
			\varphi^{n-1,d-1}\circ h_1 & * \\
	\hline
	0 & \alpha c_{n-1} \wedge \varphi^{n-1,d+1} \circ h_2
\end{array}\!\right),
	\]
	where the diagonal blocks are injective by Theorem~\ref{thmuligen}; here, we interpret $\varphi^{n-1,-1}$ and $\phi^{n-1,n+1}$ as the maps between two $0$-dimensional spaces. 
\end{proof}

\section{The AMU conjecture}
\label{amusect}
We turn now to the proofs of Theorems~\ref{AMUulige} and~\ref{AMUlige}. As will be clear, our argument is a generalization of the corresponding result of \cite{AMU}.

Recall that these are concerned with the genus $0$ quantum $\SU(N)$-representations $\rho_{N,k}^\lambda$, depending on a level $k$ and a Young diagram $\lambda \in \Lambda_{N,k}$ with at most $2$ rows. We have taken as our definition of $\rho_{N,k}^\lambda$ that based on conformal field theory, so as to directly relate the quantum representations to those discussed in the previous sections, through the following result of Kanie; see \cite{KanieCFTandBG} and \cite{UenoBook}. We refer to \cite[Sects.~4--7]{AU4} for a complete description of the relevant skein theoretically defined quantum representations and the correspondence to the ones at hand.

\begin{thm}[Kanie]
  \label{kaniethm}
	The representation $\pi_{q}^{n,d}$ of $B_n$, with $q=\exp(\frac {2\pi i}{k+N})$, is isomorphic to the quantum $SU(N)$-representation $\rho_{N,k}^\lambda$, where $\lambda$ and $d$ are related as in Section~\ref{quantumrepsintro}.
\end{thm}

\begin{lem}
  \label{kvanttilgenlem}
  Let $X_n$ denote the set of primitive $n$'th roots of unity. Then for every $z \in \U(1)$, there exist $z_n \in X_n$ such that $\lim_{n \to \infty} z_n = z$.
\end{lem}
\begin{rem}
  For $z = -1$ -- which in fact is the only case we will need in the proof of Theorem~\ref{AMUulige} -- this is \cite[Lem.~5.1]{AMU}, in which the sequence $z_n$ is constructed explicitly. We include this more general case for use in Appendix~\ref{apptorelli} below.
\end{rem}
\begin{proof}[Proof of Lemma~\ref{kvanttilgenlem}]
  Write $z = \exp(2\pi i \alpha)$, $\alpha \in [0,1)$, and let
	\begin{align*}
	  Y_n = \{ m \mid 0 < m < n,\,\, \gcd(m,n) = 1\}.
	\end{align*}	
	The largest gap between two consecutive points of $Y_n$ is bounded from above by $j(n)$, the ordinary Jacobsthal function at $n$. The result of Iwaniec \cite{Iwa} allows us to bound the size of this gap as $j(n) = O(\log^2(n))$. Thus, since $\log^2(n)/n \to 0$ as $n \to \infty$, we may choose $\alpha_n \in Y_n/n$ such that $\alpha_n - \alpha \to 0$ for $n \to \infty$. Letting $z_n = \exp(2\pi i \alpha_n) \in X_n$, we obtain the result since $z_n z^{-1} \to 1$ as $n \to \infty$.
\end{proof}
For $\phi \in \MCG(0,n)_\infty$, define $\sr_d(\phi) : \bbR \to \bbR_{> 0}$ by
\begin{align*}
	\sr_d(\phi)(x) = \sr(\diagramrep^{n,d}_{\exp(-\pi i x/4)}(\phi)),
\end{align*}
for $d \in \{0,\dots,n\}$ with $d \equiv n \mod 2$, where $\sr$ denotes the spectral radius of a linear map.
\begin{lem}
  \label{korderlem}
  Let $\phi \in \MCG(0,n)_\infty$. If $\sr_d(\phi)(x_0) > 1$ for some $x_0 \in[0,1]$, then there exists $k_0 \in \bbN$ such that the order of $\rho^{\lambda}_{N,k}(\phi)$ is infinite for all $k > k_0$, where $\lambda$ and $d$ are related as in Section~\ref{quantumrepsintro}.
\end{lem}
\begin{proof}
  It follows from Theorem~\ref{kaniethm} and the relations described in Section~\ref{jonesdefsect} that it suffices to show that for some $k_0$, the order of $\diagramrep_A^{n,d}(\phi)$ is infinite for $q = A^4 = \exp(2 \pi i/(k+N))$ for all $k$ with $k > k_0$. On the other hand, the order of $\diagramrep_A^{n,d}(\phi)$ at a general primitive $4(k+N)$'th root of unity $ A = \exp(\frac{2\pi i l}{4(k+N)})$ is independent of $l$.

	Now, apply Lemma~\ref{kvanttilgenlem} to $q = \exp(-\pi i x_0)$ to obtain primitive $4(k+N)$'th roots of unity $A_k$ with $A_k^4 \to q$. Since $\sr_d(\phi)$ is continuous, $\sr(\diagramrep^{n,d}_{A_k}(\phi)) > 1$, for all sufficiently large $k$, and a linear map having an eigenvalue of absolute value greater than $1$ necessarily has infinite order.
\end{proof}

To show that a given pseudo-Anosov $\phi \in \MCG(0,n)_\infty$ has infinite order for all but finitely many levels, it thus suffices to find $x_0$ as above. As we will see below, by Theorems~\ref{thmuligen} and \ref{thmligengenyoung}, this is possible for those pseudo-Anosovs whose stretch factors are given by their action on the homology of the double cover. In case $d = n-2$, the claims will follow immediately from Lemma~\ref{korderlem} but in general, we will need the following remarks on the actions of surface diffeomorphisms on wedge products of homology.

\begin{remark}
  \label{homologibem}
	Let $f_*$ be the action on $H_1(\Sigma, \bbC)$, where $\Sigma$ is a surface of genus $g$ with $0$ or $1$ boundary components, induced by a diffeomorphism $f$. Let $\lambda_1, \dots. \lambda_{2g}$ be the diagonal entries in the Jordan normal form of the matrix for $f_*$. Then the action of $f$ on $\Lambda^l H_1(\Sigma, \setC)$ will have an eigenvector of eigenvalue $\prod_{i\in I} \lambda_i$ for any subset $I \subseteq \{1:2g\}$ of size $l$, given by wedging together the corresponding vectors in the Jordan normal form, in such a way that a non-eigenvector is only included if all the preceding vectors in its block are also included. As all of these eigenvectors are of the form $\bigwedge_{i=1}^l v_i$, they can not be in the subspace $\omega \wedge \Lambda ^{l-2} H_1(\Sigma, \bbC)$ if $l\leq g$, and therefore they define eigenvectors of the same eigenvalue in the quotient $\Lambda^{l} H_1(\Sigma, \bbC)/ \omega \wedge \Lambda^{l-2} H_1(\Sigma, \bbC)$.
	
	Assume now that the action on $H_1(\Sigma,\bbC)$ has an eigenvalue with absolute value strictly greater than $1$. The action is symplectic, so the eigenvalues come in pairs $\lambda,\lambda^{-1}$, and there must be at least $g$ columns in the Jordan normal form having diagonal entry with absolute value at least $1$. Therefore the action on $\Lambda^l H_1(\Sigma,\bbC)$ has an eigenvalue with absolute value greater than $1$ for $l\leq g$, and by the above considerations, so does the action on $\Lambda^{l} H_1(\Sigma, \bbC)/ \omega \wedge \Lambda^{l-2} H_1(\Sigma, \bbC)$.

\end{remark}
\begin{prop}
  \label{uendeligordenprop}
  Let $f \in \MCG(0,n)_\infty$, and assume that the action $\hat{f}_*$ of $\hat{f} = \Psi(f)$ on $H_1(\Sigma_g^m)$ has spectral radius strictly greater than $1$. Then the order of $\rho^{\lambda}_{N,k}(f)$ is infinite for all but finitely many levels $k$.
\end{prop}
\begin{proof}
  As noticed above, this follows immediately from our main theorems in case $d = n-2$, by application of Lemma~\ref{korderlem} with $x_0$ = 1, since $\sr_d(f)(-1) = \sr(\hat{f}_*)$.
  
  For general $d$, we also need to ensure that an appropriate eigenvector -- of eigenvalue with absolute value strictly greater than $1$ -- is actually contained in the image of the morphism of representations. This, on the other hand, is guaranteed by Remark~\ref{homologibem} in case $n$ is odd.
  
Assume now that $n$ is even, $d < n-2$. Since $\hat{f}$ preserves the boundaries pointwise, it defines a diffeomorphism of $\Sigma_{g+1}^1$, obtained by gluing to $\Sigma_g^2$ a pair of pants. Denote by $\iota$ the induced map on (wedge products of) homology. On the level of diagrams, this corresponds to the inclusion $V^{n,d} \hookrightarrow V^{n+1,d+1}$ obtained by adding to a diagram the point $n+1$ and connecting it to $\infty$. This of course corresponds to the decomposition $V^{n+1,d+1} = V^{n,d} \oplus V^{n,d+2}$ as a vector space, as described in the proof of Theorem~\ref{thmligengenyoung} (but note that $n+1$ is now odd); that is, $V^{n,d+2}$ is spanned by diagrams where $n+1$ is not connected to $\infty$. Now, even though the action of $f$ does not preserve the decomposition, it is clearly block triangular.

With these identifications, we have a diagram

  \centerline{\xymatrix{
    V^{n,d} \ar[r] \ar[d]_-{\phi^{n,d}} & V^{n+1,d+1} \ar[d]^-{\phi^{n+1,d+1}}&\\
    \Lambda^l H_1(\Sigma_g^2,\bbC) \ar[r]^-{\iota} & \Lambda^l H_1(\Sigma_{g+1}^1,\bbC) \ar[r]& \Lambda^l H_1(\Sigma_{g+1}^1,\bbC)/\omega \wedge \Lambda^{l-2} H_1(\Sigma_{g+1}^1,\bbC),
  }}
which is commutative up to a power of $-i$. We wish to show that the action of $f$ on $V^{n,d} \subseteq V^{n+1,d+1}$ contains eigenvectors of the appropriate eigenvalues. Suppose that the eigenvalues of the action of $f_*$ on $H_1(\Sigma_g^2,\bbC)$, counted with algebraic multiplicity, have absolute values 	
\[
	(x_1,\dots,x_m,1,1,\dots,1,x_m^{-1},\dots,x_1^{-1})
\]
with $x_i > 1$ for all $i$ so that as before, $m \geq 1$.	Consider first the case $l \leq m$. As in Remark~\ref{homologibem}, the action of $\hat{f}_*$ on $\Lambda^l H_1(\Sigma_g^2,\bbC)$ has an eigenvector $v = v_1 \wedge \dots \wedge v_l$ whose eigenvalue has absolute value $x = x_1 \cdots x_l$. Now, $\iota(v)$ is an eigenvector for the induced action with the same eigenvalue (up to the same root of unity), and it follows from the case of $n$ odd that $\Im(\phi^{n+1,d+1})$ contains an eigenvector which has the same eigenvalue as $\iota(v)$. Moreover, since the eigenvectors arising from $V^{n,d+2}$ all have absolute value strictly less than $x$ (as we take only the $(l-1)$'st wedge products), we obtain the desired eigenvector of $\diagramrep^{n,d}_{\exp(-\pi i/4)}(f)$. The conclusion now follows as in the case of odd $n$.
	
	The case $l > m$ is similar but involves also a small combinatorial exercise as in this case, there may also be eigenvectors coming from $V^{n,d+2}$ with eigenvalue of absolute value $x$. Let $d_x^{n,d}$ be the sum of the algebraic multiplicities of eigenvalues of absolute value $x$ of the action of $f$ on $V^{n,d}$. We claim that for $n$ even,
	\begin{align}
	  \label{binomiallighed}
	  d_x^{n,d} = \binom{2g+1-2m}{l-m} - \binom{2g+1-2m}{l-m-2}.
	\end{align}
	We appeal again to the decomposition used above, as a simple extension of the argument from Remark~\ref{homologibem} shows that
	\[
	  d_x^{n+1,d+1} = \binom{2g+2-2m}{l-m} - \binom{2g+2-2m}{l-m-2}.
	\]
	Now, since $d_x^{n+1,d+1} = d_x^{n,d} + d_x^{n,d+2}$, equation \eqref{binomiallighed} follows by induction on $l$, starting at $l = m$, by using well-known recursive formulas for binomial coefficients. Since $d_x^{n,d} > 0$, this completes the proof.
\end{proof}
\begin{proof}[Proof of Theorem~\ref{AMUulige}]
  Let $f \in \MCG(0,n)_\infty$ be a homological pseudo-Anosov, let $\tilde{f}$ denote its image in $\MCG(0,n+1)$, and let $\hat{f} \in \MCGrand(g,m)$ denote the image of $f$ under the Birman--Hilden map used in Theorems~\ref{thmuligen} and \ref{thmligengenyoung} with the appropriate values of $g$ and $m$. Everything has been set up so that $\hat{f}$ is a pseudo-Anosov of $\Sigma_g^m$ with the same stretch factor as $\tilde{f}$, and that, moreover, $\hat{f}$ has \emph{orientable} invariant foliations. This follows by the exact same reasoning as in the similar setup in Theorem~5.1 of \cite{BandBoyland}. In short: the orientability of a foliation is determined by the vanishing of its associated orientation homomorphism, and this on the other hand is ensured by the assumptions on the degrees of the singularities.
  
  Now, the stretch factor of any pseudo-Anosov with orientable invariant foliations is simply the spectral radius of its action on homology which is therefore strictly greater than $1$. This is a well-known result and in fact a criterion for having orientable foliations; see e.g. \cite[Lemma~4.3]{BandBoyland} and the discussion preceding it.
  
  Now the claim follows directly from Proposition~\ref{uendeligordenprop}.
\end{proof}
\begin{proof}[Proof of Theorem~\ref{AMUlige}]
  Consider now the case where $N = 2$, $n$ is even, and $\lambda$ is the empty diagram. Here, the level $k$ quantum $\SU(2)$-representation, rescaled on each generator by a suitable $k$-dependent root of unity, defines a representation of $\qMCG(0,n)$, equivalent by construction to $\rho_{2,k}$ (see \cite[Sect.~10]{JonHecke}). As multiplication by a root of unity does not change whether or not the order of a linear map is finite or infinite, we obtain from Theorem~\ref{AMUulige} the claimed result.
\end{proof}
\begin{rem}
  More generally, in \cite[Sect.~10]{JonHecke}, Jones finds that his representations may be tweaked by roots of unity to descend to the mapping class groups of spheres whenever the associated Young diagram $\lambda$ is rectangular. Thus, the same is of course true for the quantum $\SU(N)$-representations $\rho_{N,k}^\lambda$. One could therefore proceed as in \cite[Sect.~4]{AMU}, define new quantum representations for mapping class groups of punctured spheres, and immediately obtain a version of Theorem~\ref{AMUlige} for those.
\end{rem}

\begin{example}
  At this point it may be worth remarking that examples of homological pseudo-Anosovs are plenty.
  \begin{itemize}
  \item In \cite[App. A]{LT}, the authors provide several examples of homological (as well as non-homological) braids, some of which we shall return to in the appendix below.
  \item On a closed torus, the stretch factor of a pseudo-Anosov is always given by its action on homology, and so we recover the main result of \cite{AMU}.
  \item As noted in the proof of Theorem~\ref{AMUulige}, \cite[Lemma~4.3]{BandBoyland} tells us that if the spectral radius of the action of a pseudo-Anosov on homology equals its stretch factor, the pseudo-Anosov must necessarily have orientable invariant foliations, and so we can appeal directly to any of the existing homological constructions of pseudo-Anosovs to obtain interesting examples. One such family of examples arises as a special case of the pseudo-Anosovs described in \cite{Pen} (see also \cite[Sect.~5]{CB}) on the level of the covering surfaces, which -- passing through the Birman--Hilden homomorphism -- may be described as follows. Suppose that $n$ is even. Then any word in the generators $\sigma_1,\dots,\sigma_{n-1}$ such that the sign of the exponents for odd-indexed generators all agree, such that exponents of even-indexed generators all have the opposite sign as the odd-indexed ones, and such that each generator appears at least once, is a homological pseudo-Anosov.
  \end{itemize}
\end{example}

\subsection{Determining stretch factors}
In \cite[Cor.~5.8]{AMU}, the authors go on to show that the stretch factor of any given pseudo-Anosov of a four times punctured sphere may be obtained as limits of eigenvalues of the quantum representations of $\phi$. The analogous statement in our general case is the following.
\begin{cor}
  \label{sfcor}
  For any homological pseudo-Anosov $\phi \in \MCG(0,n)_\infty$ and Young diagram $\lambda$ as in Section~\ref{quantumrepsintro}, there exist eigenvalues $\lambda_k, \tilde{\lambda}_k$ of $\rho_{N,k}^\lambda(\phi)$ such that $\sqrt{\abs{\lambda_k \tilde{\lambda}_k}}$ tends to the stretch factor of $\phi$ as $k \to \infty$.
\end{cor}
\begin{proof} 
	The statement follows from the proof of Proposition~\ref{uendeligordenprop} by continuity in $A$ of the eigenvalues of $\diagramrep^{n,d}_A(\phi)$, as these include at $q = A^4 = -1$ the values $\tau \mu$ and $\tau \mu^{-1}$, where $\abs{\tau}$ is the stretch factor of $\varphi$ (and likewise, $\mu$ is a product of eigenvalues of the induced action on homology).
\end{proof}

\appendix
\section{Examples, experiments, and observations}
\label{apptorelli}
\subsection{A pseudo-Anosov in the Torelli group}
The preceding discussion begs the question of the behaviour of those pseudo-Anosov mapping classes for which the action on homology contains no information. In particular, we could consider examples of pseudo-Anosovs acting trivially on the homology of the double cover. That pseudo-Anosovs in the Torelli group, i.e. the kernel of the action of $\Gamma_g$ on $H_1(\Sigma_g)$, exist for $g \geq 2$ was shown in \cite{ThAnosov}. One concrete such element was constructed by Brown \cite[App.]{BroConst} for $g = 2$, and appealing once again to the setup of the Birman--Hilden theorem, this allows us to obtain the desired braid $\phi \in B_6$.

More precisely, define
\begin{align}
  \label{BrownpA}
	\phi = \eta \xi^{-1}
\end{align}
explicitly in terms of the standard braid group generators by
\begin{align*}
  \xi &= \gamma \eta \gamma^{-1}, \\
	\eta &= \sigma_4^2\sigma_5\sigma_4^2\sigma_5^2\sigma_4^2\sigma_5\sigma_4^2, \\
	\delta &= \sigma_4\sigma_5\sigma_3\sigma_2\sigma_1\sigma_2^{-1}\sigma_3^{-1}\sigma_5^{-1}\sigma_4^{-1}, \\
	\gamma &= \delta \sigma_3 \sigma_2^{-1}\sigma_1^{-1}\delta \sigma_1\sigma_2\sigma_3^{-1}\delta^{-1}.
\end{align*}

As we have seen, to show that a pseudo-Anosov $\psi \in B_n$ has infinite order in the quantum representations for all but finitely many levels, it suffices to show that the function $\sr_d(\psi) : [0,1] \to \bbR_{> 0}$ given by
\begin{align*}
	 \sr_d(\psi)(x) = \sr(\diagramrep_{\exp(-\pi i x/4)}^{n,d}(\psi))
\end{align*}
is greater than $1$ for some value $x \in [0,1]$. The result of evaluating this function for the pseudo-Anosov $\phi$ of \eqref{BrownpA} is shown\footnote{This figure, and all other figures in this section, were created using Sage~5.13 and Mathematica~9.0.} in Figure~\ref{srBrownpA}. We see for instance that $\sr_0(\phi)(x) = 1$ for $x \in [0,\tfrac{1}{2}] \cup \{1\}$ but that this is not the case in general.

\begin{figure}[h]
  \centering
	\includegraphics[scale=0.8]{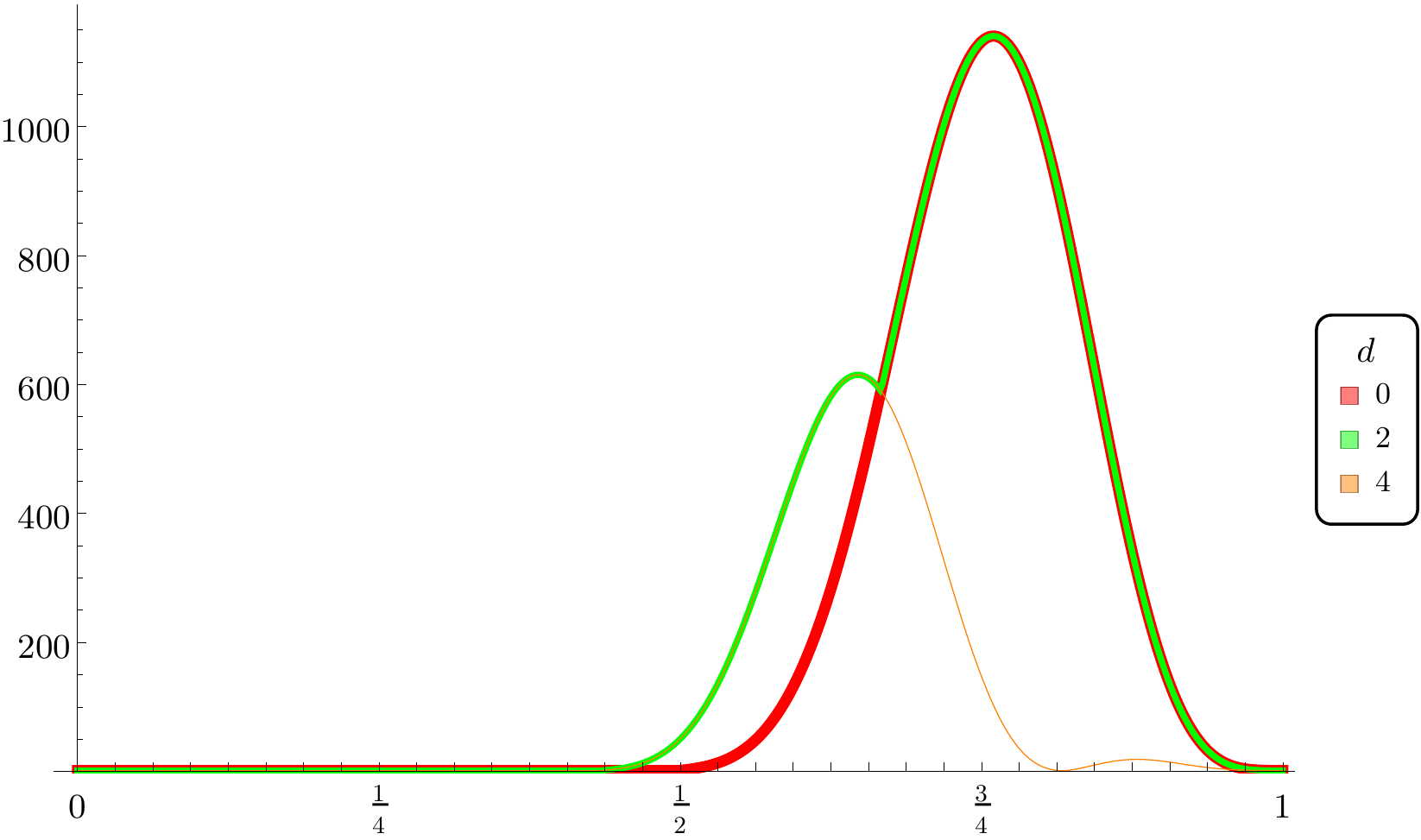}
	\caption{Graphs of $\sr_d(\phi)(x)$ for $x \in [0,1]$.}
	\label{srBrownpA}
\end{figure}

\begin{prop}[Computer assisted]
  There exist pseudo-Anosovs in $\ker \diagramrep^{n,d}_{\exp(-\pi i/4)}$ that have infinite order in the corresponding $\rho_{N,k}^\lambda$ for all but finitely many $k$.
\end{prop}

Previously, we saw that $\sr$ naturally determines the stretch factors of homological pseudo-Anosov but it is not clear in which sense, if any, this is true for the mapping class $\phi$ above.

\subsection{Infinite order at low levels}
Other examples exhibiting similar behaviours include the stretch factor minimizing pseudo-Anosov braids of \cite{LT}; plots of $\sr$ are shown for these in Figure~\ref{masserafplots}. In each of these, we let $\psi_n = \sigma_1 \cdots \sigma_n$, and the dashed line indicates the stretch factor of the pseudo-Anosov in question. In particular, the braid in Figure~\ref{LTpa4} provides an interesting example: it has finite $k$-dependent order in $\rho_{N,k}^\lambda$ for every level $k$ with $N > 2$ and $\lambda = {\fontsize{4pt}{12pt}\selectfont {\yng(2,2)}}$. We remark that in \cite[Conjec.~2.4]{AMU}, the authors do not mention any possible assumptions on the labelling in the relevant quantum representations. 
\begin{figure}[h]
  \centering
  \begin{subfigure}[c]{0.47\textwidth}
    \centering
    \includegraphics[scale=0.45]{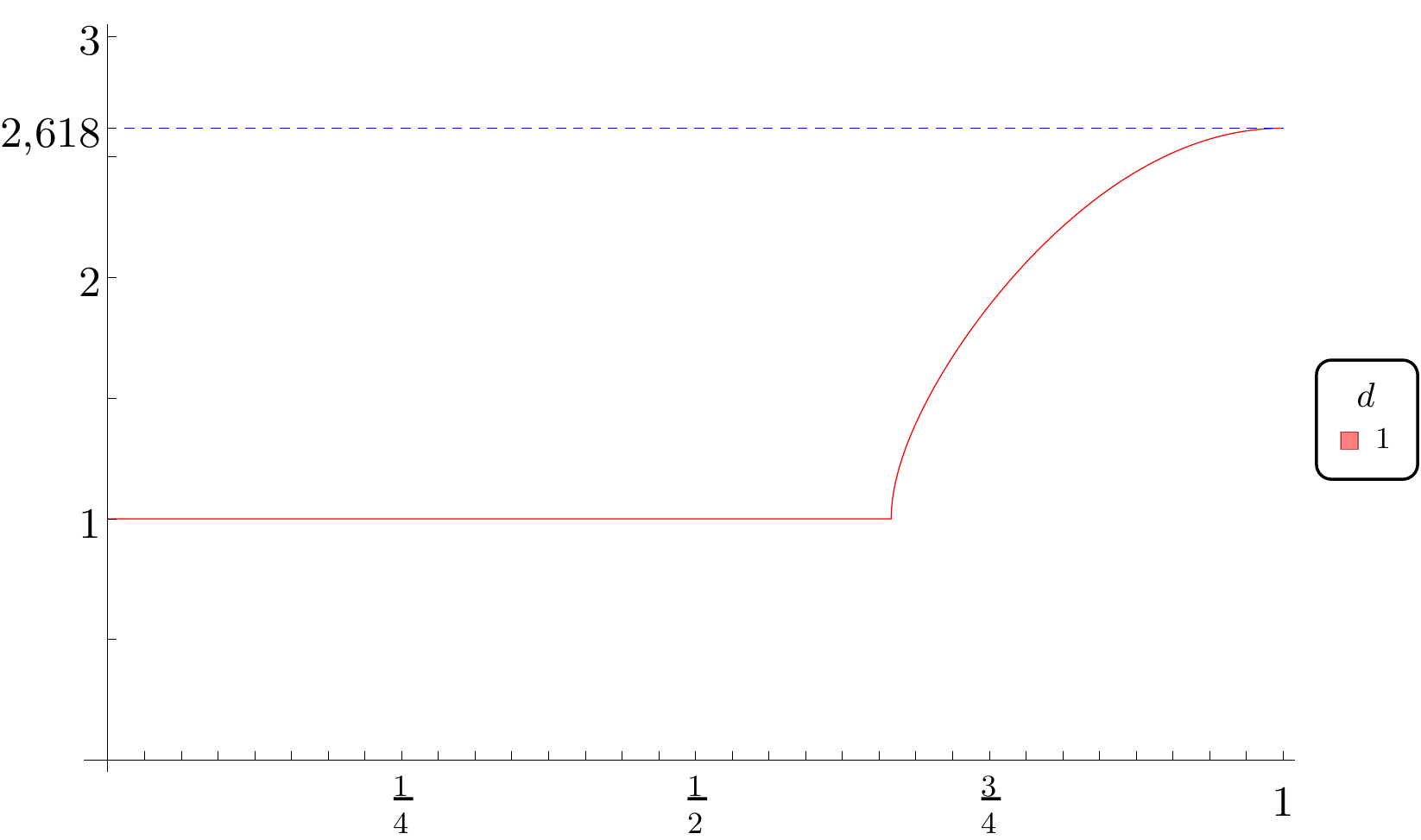}
    \caption{$\sigma_1\sigma_2^{-1} \in B_3$.}
    \label{LTpa3}
  \end{subfigure}
  \hfill
  \begin{subfigure}[c]{0.47\textwidth}
    \centering
    \includegraphics[scale=0.45]{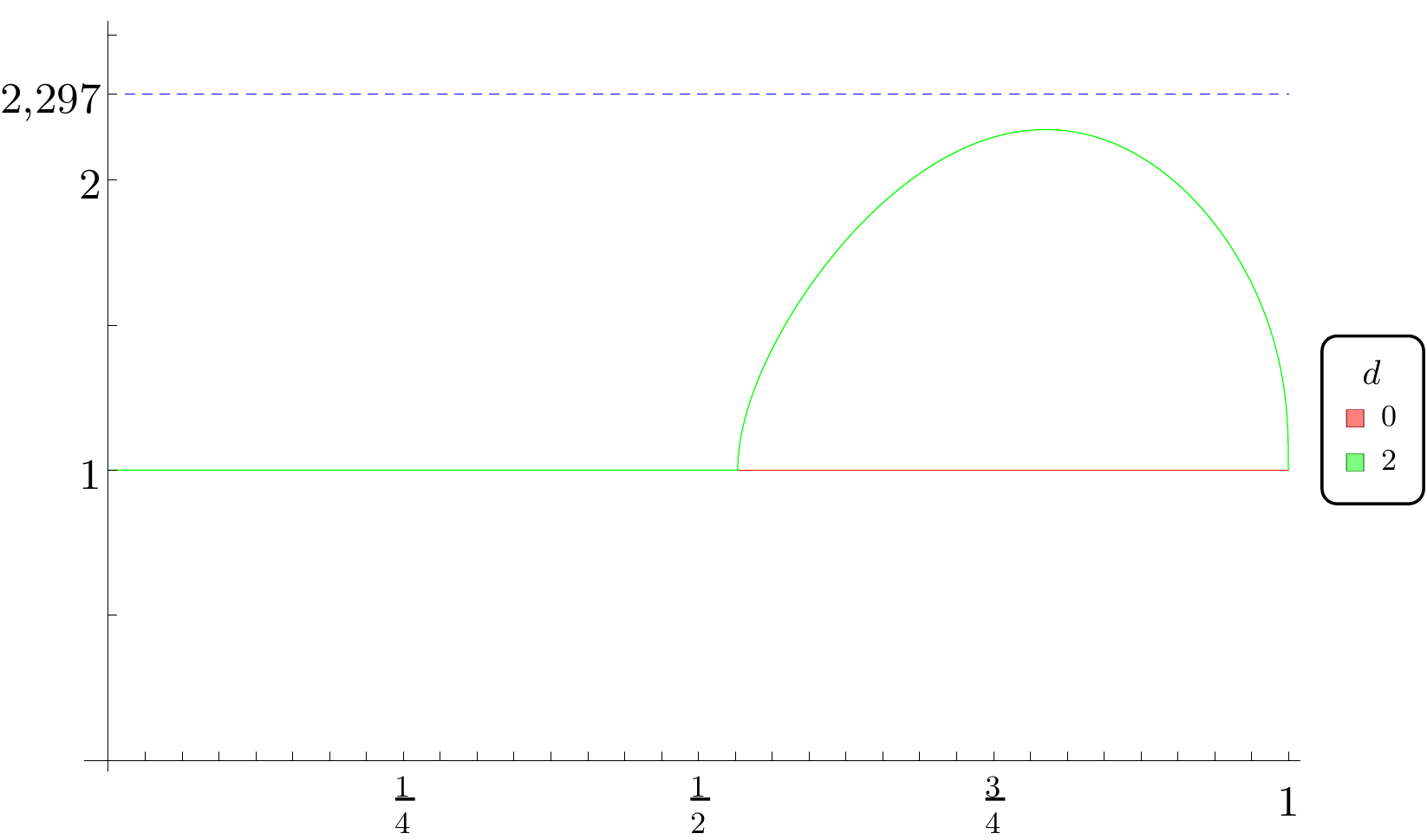}
    \caption{$\sigma_1\sigma_2\sigma_3^{-1} \in B_4$.}
    \label{LTpa4}
  \end{subfigure}
  \\[0.5cm]
  \begin{subfigure}[c]{0.47\textwidth}
    \centering
    \includegraphics[scale=0.45]{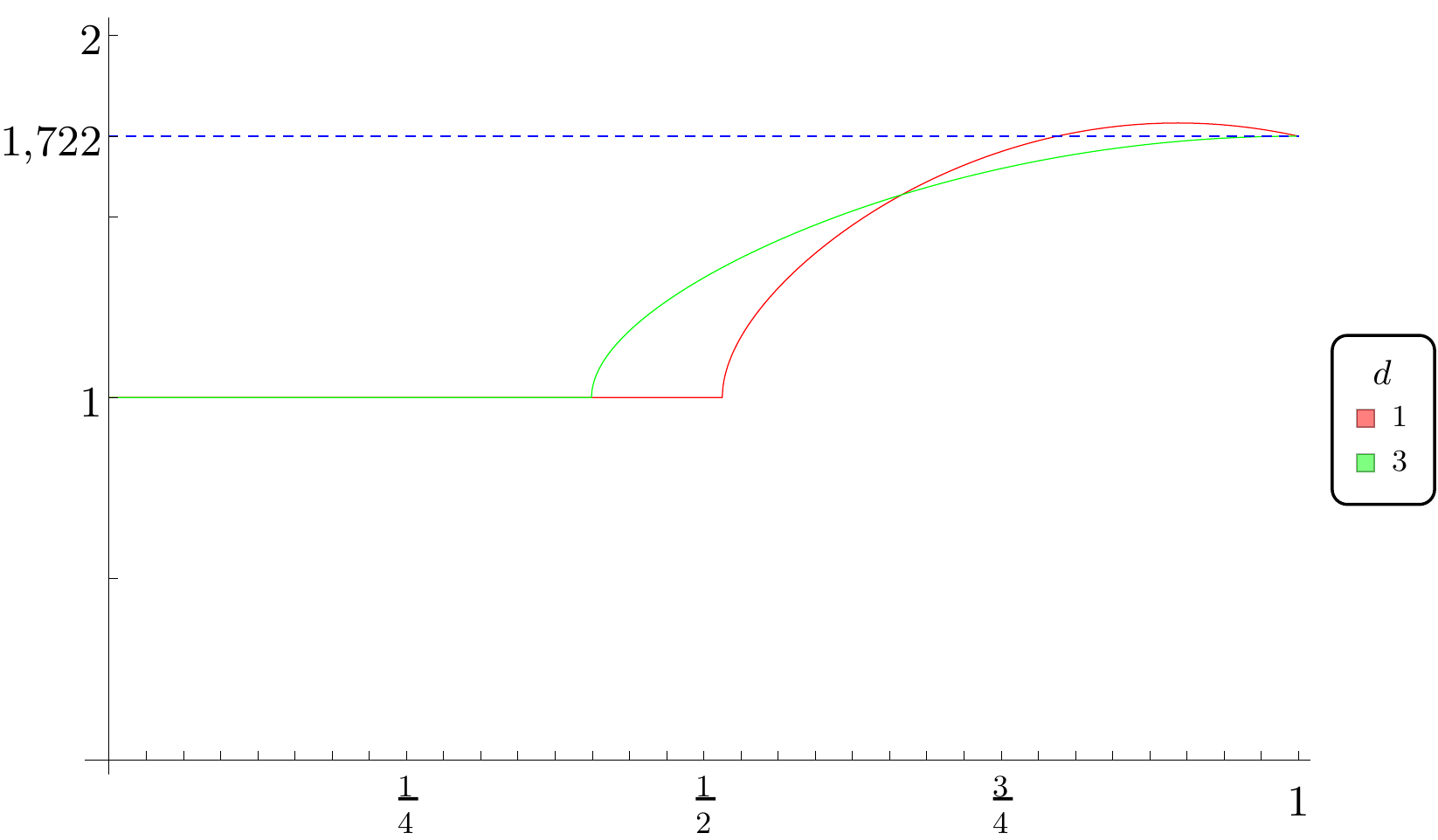}
    \caption{$\psi_3^2\sigma_4\sigma_3^{-1} \in B_5$.}
    \label{LTpa5}
  \end{subfigure}
  \hfill
  \begin{subfigure}[c]{0.47\textwidth}
    \centering
    \includegraphics[scale=0.45]{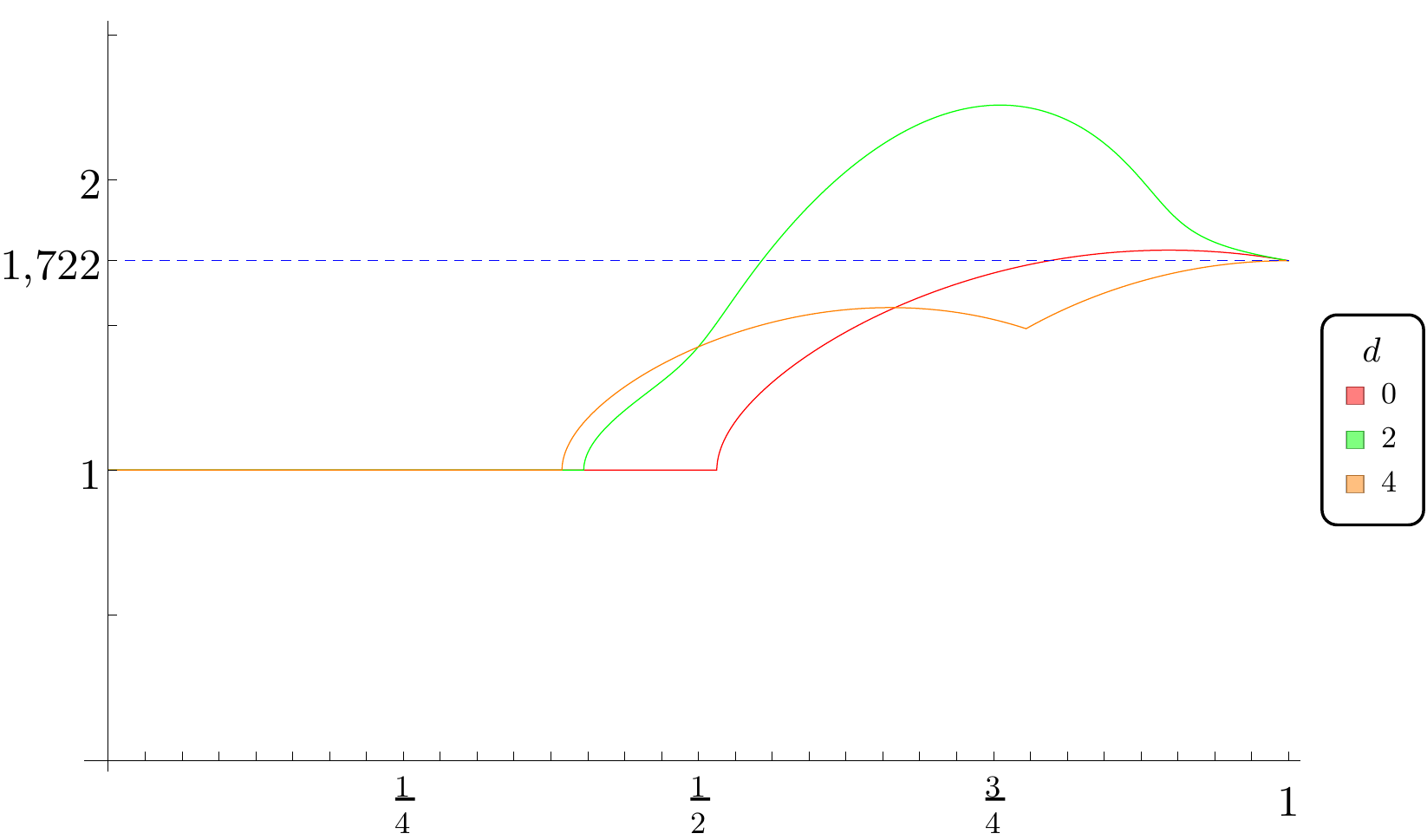}
    \caption{$(\sigma_2\sigma_1)^2\psi_5^2 \in B_6$.}
    \label{LTpa6}
  \end{subfigure}
  \\[0.5cm]
  \begin{subfigure}[c]{0.47\textwidth}
    \centering
    \includegraphics[scale=0.45]{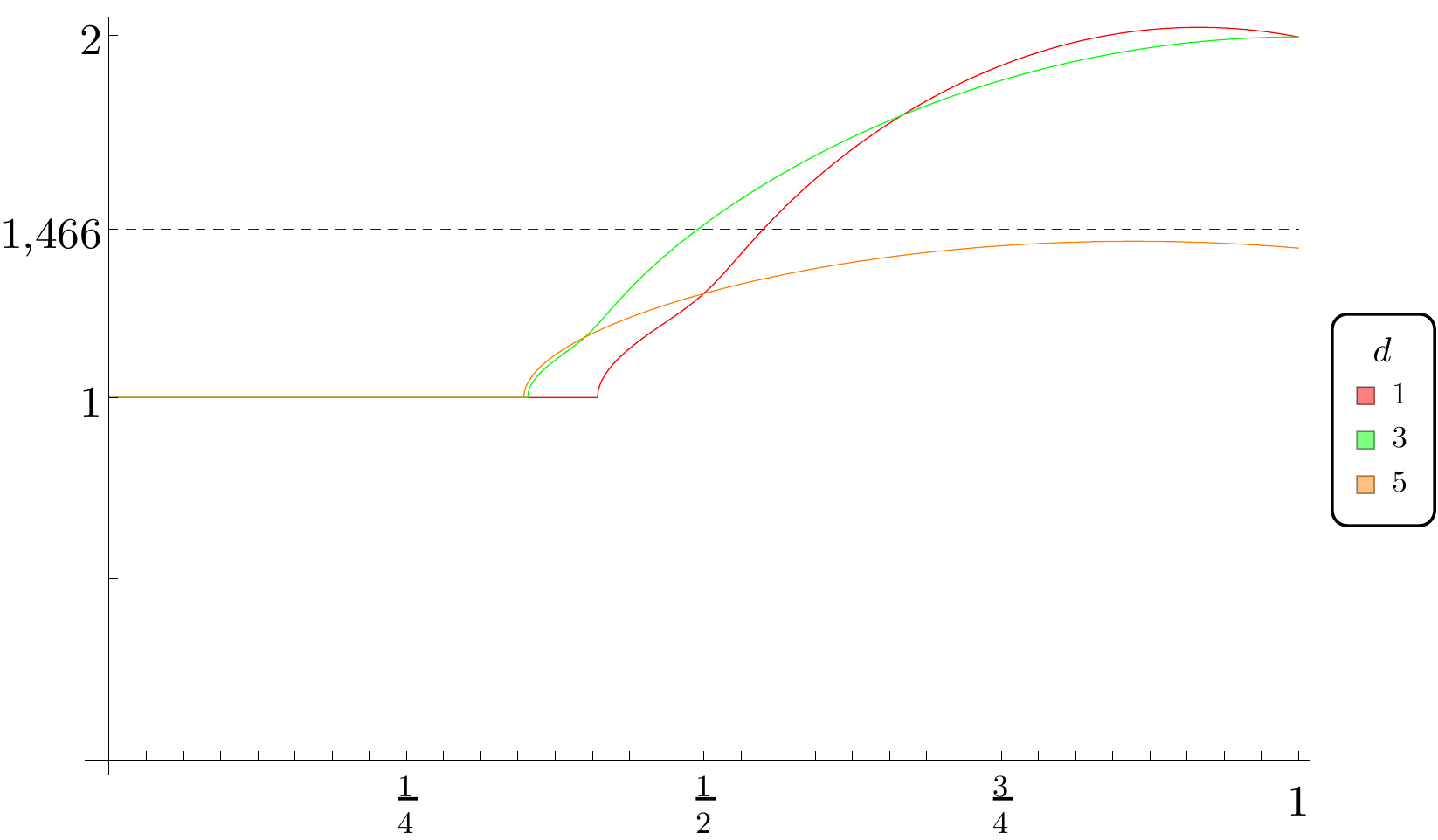}
    \caption{$\sigma_4^{-2}\psi_6^2 \in B_7$.}
    \label{LTpa7}
  \end{subfigure}
  \hfill
  \begin{subfigure}[c]{0.47\textwidth}
    \centering
    \includegraphics[scale=0.45]{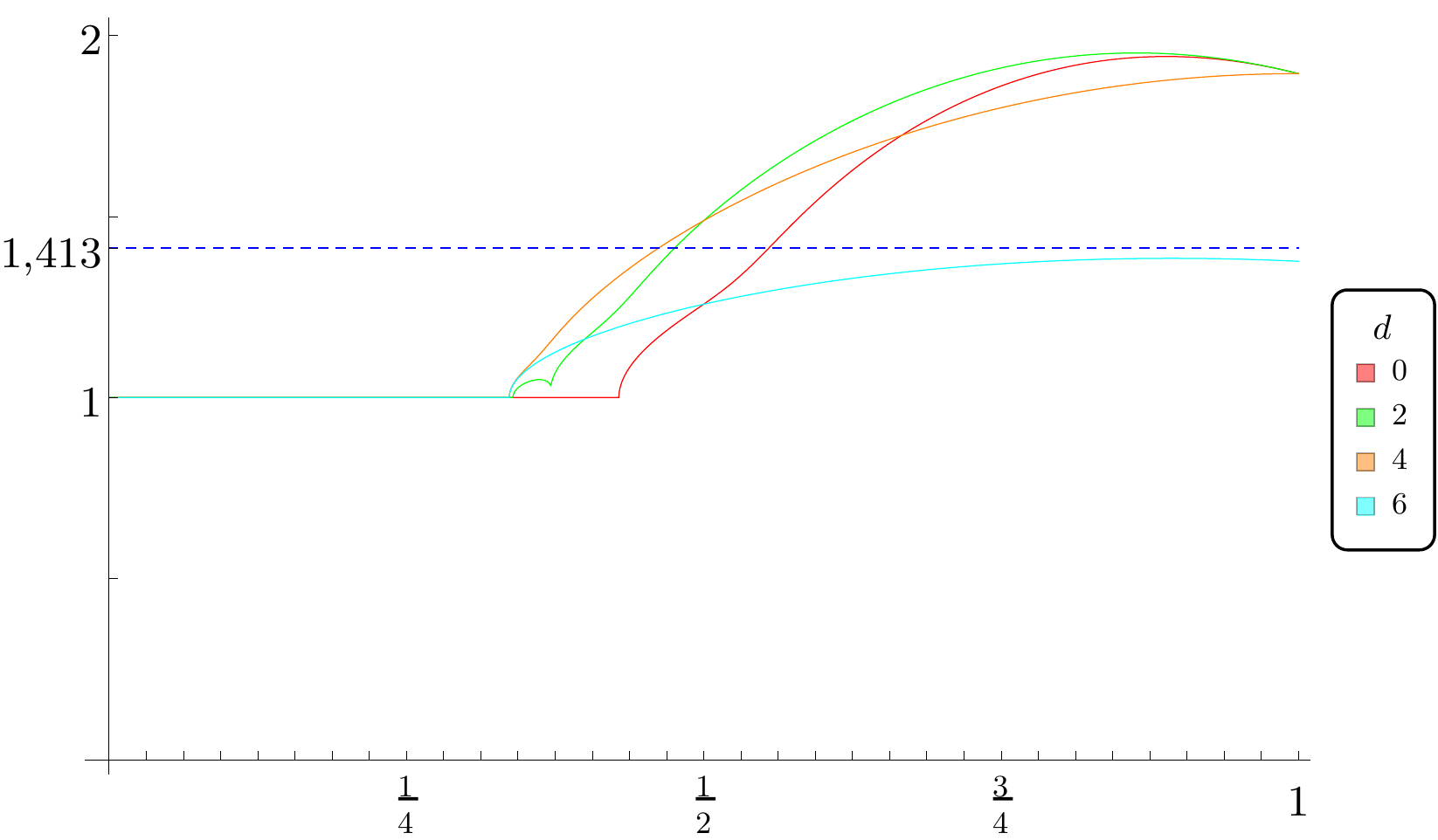}
    \caption{$\sigma_2^{-1}\sigma_1^{-1}\psi_7^5 \in B_8$.}
    \label{LTpa8}
  \end{subfigure}
  \caption{Plots of spectral radii for braids with small stretch factors.}
  \label{masserafplots}
\end{figure}

Notice also that these graphs allow us to deduce more precisely at \emph{which} levels $k$, the orders of the mapping classes in question are infinite. For example, in Figure~\ref{uendeligorden}, we consider again the braid of Figure~\ref{LTpa4}, now viewed as an element of $B_6$ and highlight the possible specializations of $A$ for the level $k$ quantum $\SU(2)$-representation when $k = 5, \dots, 12$.

\begin{figure}[h]
  \centering
	\includegraphics[scale=0.8]{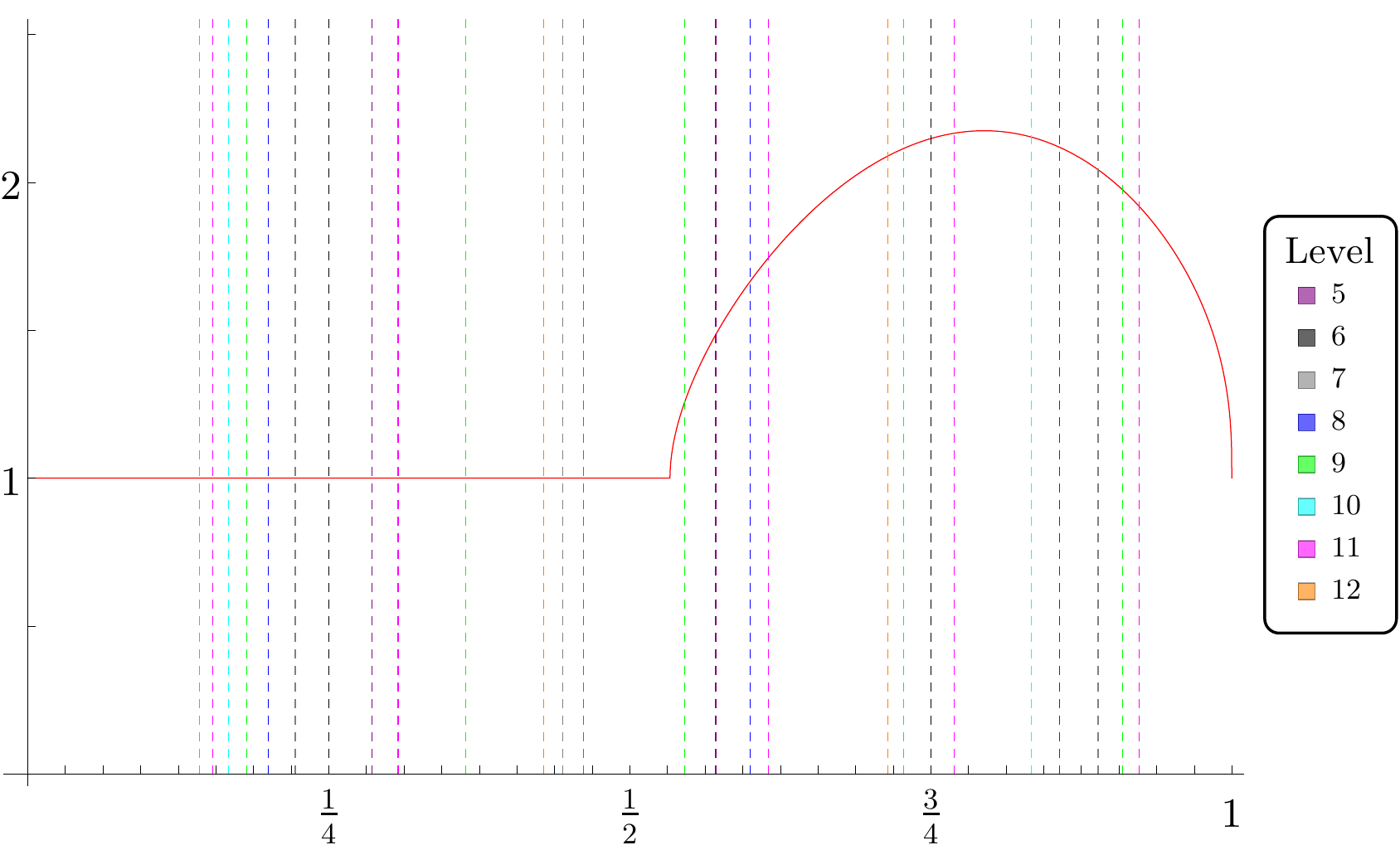}
	\caption{The graph of $\sr_0$ for $\sigma_1\sigma_2\sigma_3^{-1} \in B_6$.}
	\label{uendeligorden}
\end{figure}

From this figure alone, we obtain the following.
\begin{prop}
  \label{inforder8}
  The level $8$ quantum $\SU(2)$-representation, with $\lambda$ the empty Young diagram, of the mapping class $\sigma_1\sigma_2\sigma_3^{-1} \in B_6$ is infinite.
\end{prop}
\begin{proof}
  Whereas the claim follows from Figure~\ref{uendeligorden}, let us also note that one finds by explicit calculation -- requiring no electronic aid -- that for a particular ordering in our preferred basis,
	\begin{align*}
	  \diagramrep_A^{6,0}(\sigma_1\sigma_2\sigma_3^{-1}) = A \begin{pmatrix} 0 & A^{-6} & 0 & -1 & A^2 \\ A^2 & -1 + A^{-4} & -1 & 0 & 0 \\ 0 & 0 & 0 & A^{-6} - A^{-2} & 1 \\ 0 & 0 & -A^{-2} & -1 + A^{-4} & A^2 \\ 0 & 0 & 0 & -A^2 & A^4 \end{pmatrix}.
	\end{align*}
	Substituting the primitive $4\cdot(8+2)$'th root of unity $A = \exp(2\pi i \tfrac{3}{40})$, one obtains that
  \begin{align*}
	  \sr(\diagramrep_A^{6,0}(\sigma_1\sigma_2\sigma_3^{-1})) \approx 1{,}665 > 1.
	\end{align*}
\end{proof}
While we have not considered higher genus quantum representations of mapping class groups in the present paper, let us notice that it follows from Proposition~\ref{inforder8} and a standard factorization argument that the quantum representations of closed surfaces of genus $g > 2$ have infinite image at level $k = 8$; this extends an earlier result due to Masbaum \cite{Masinf} (which is indeed the reason we considered this particular level). In fact, at level $k = 8$, the image of the quantum $\SU(2)$-representation obtained from the sphere with four marked points is known to be finite by the result of \cite{LPS}, and so one can not use the Jones representations to obtain a similar result for $g = 2$, $k = 8$. However, the second author has found -- again aided by computer calculation -- example mapping classes whose quantum representations have infinite order for $g = 2$, $k = 8$, by embedding in the genus $2$ surface a one-holed torus rather than a four-holed sphere; see e.g. \cite[Prop.~4.16]{JoerThesis}, \cite[p.~49]{Joer}.

Now, the question about the size of the image is still open for levels $k = 1$, $2$, and $4$, and casting a sidelong glance at Figures~\ref{LTpa7} and \ref{LTpa8}, noting that at the fourth root of unity $q = A^4 = i$, the function $\sr$ is greater than $1$, one could be inclined to believe that these provide infinite order examples for $k = 2$. Recall however that these specializations do not correspond to the level $2$ quantum representations in which one finds that the two elements have orders $28$ and $16$ respectively. Indeed, in the cases $k = 1$, $2$, and $4$, the only possible specializations of values of $q$ allowed by skein theory are $q = \exp(2\pi i/(k+2))$. The corresponding representations are known to be unitary and so no information may be gained from considering spectral radii.

\subsection{Bigelow's element of the kernel of the Burau representation}
As a final remark, the connection between the various $\sr_d$ is an interesting question. As a single example, Figure~\ref{bigelow} shows $\sr_d(\psi)$ for Bigelow's \cite{Big} element $\psi \in \ker(\pi^{5,3}_q)$, given explicitly by
\begin{align*}
	\psi = [\psi_1^{-1}\sigma_4\psi_1,\psi_2^{-1}\sigma_4\sigma_3\sigma_2\sigma_1^2\sigma_2\sigma_3\sigma_4\psi_2],
\end{align*}
where
\begin{align*}
	\psi_1 &= \sigma_3^{-1}\sigma_2\sigma_1^2\sigma_2\sigma_4^3\sigma_3\sigma_2, \\
	\psi_2 &= \sigma_4^{-1}\sigma_3\sigma_2\sigma_1^{-2}\sigma_2\sigma_1^2\sigma_2^2\sigma_1\sigma_4^5.
\end{align*}
\begin{figure}[h]
  \centering
	\includegraphics[scale=0.8]{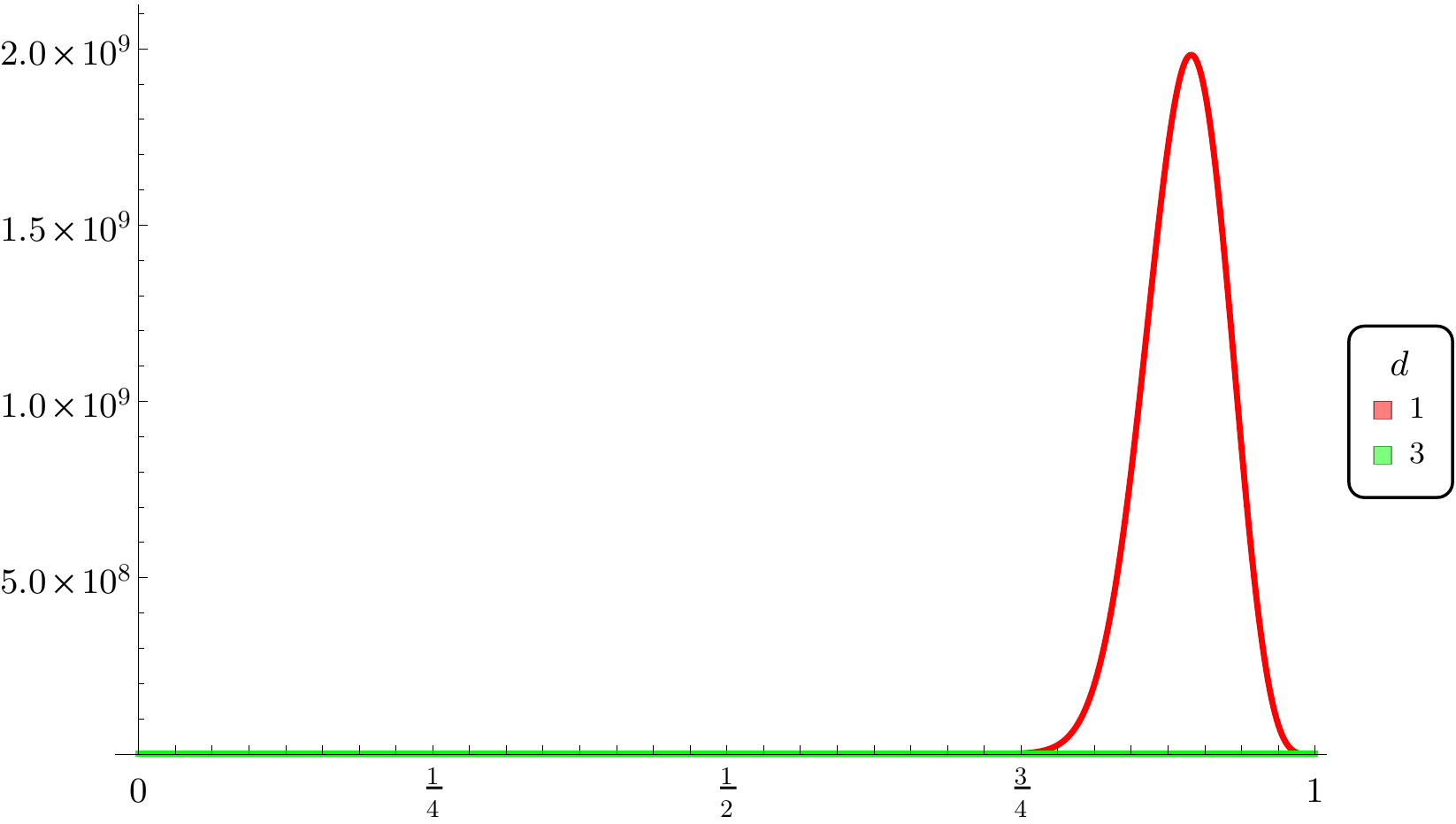}
	\caption{The functions $\sr_d(\psi)(x)$, $x \in [0,1]$, for a $\psi \in \ker(\pi^{5,3}_q)$.}
	\label{bigelow}
\end{figure}

\clearpage
\bibliographystyle{is-alpha}
\bibliography{references}
\end{document}